\documentclass[a4paper, 11pt]{article}

\usepackage[noend]{algpseudocode}
\usepackage{algpascal}
\usepackage[linesnumbered,ruled]{algorithm2e}

\usepackage[english]{babel}
\usepackage[utf8x]{inputenc}
\usepackage{amsmath}
\usepackage{graphicx}
\usepackage[colorinlistoftodos]{todonotes}

\usepackage{amsthm}
\usepackage{amsmath}
\usepackage{amsfonts}

\usepackage[labelfont=bf]{caption}

\usepackage[numbers]{natbib}
\usepackage{url}

\usepackage{caption}
\usepackage{subcaption}

\newcommand{\R}{R}

\newtheorem*{remark}{Remark}
\newtheorem*{example}{Example}
\newtheorem{theorem}{Theorem}

\newtheorem{lemma}[theorem]{Lemma}
\newtheorem{cor}{Corollary}[theorem]

\title{A primal-dual interior point method for a novel type-2 second order cone optimization problem}

\author{Md Sarowar Morshed \thanks{Department of Mechanical $\&$ Industrial Engineering, Northeastern University, Boston, MA 02115, USA 
 }
\and Chrysafis Vogiatzis \thanks{Department of Industrial $ \& $ Enterprise Systems Engineering, University of Illinois at Urbana-Champaign, Urbana, IL 61801, USA}
\and Md Noor-E-Alam \thanks{Corresponding Author: mnalam@neu.edu} \footnotemark[2]}

\begin{document}
\maketitle

\section*{Abstract}

In this paper, we define a new, special second order cone as a type-$k$ second order cone. We focus on the case of $k=2$, which can be viewed as SOCO with an additional {\em complicating variable}. For this new problem, we develop the necessary prerequisites, based on previous work for traditional SOCO. We then develop a primal-dual interior point algorithm for solving a type-2 second order conic optimization (SOCO) problem, based on a family of kernel functions suitable for this type-2 SOCO. We finally derive the following iteration bound for our framework:
\[\frac{1}{\theta \kappa \gamma} \left[2N \psi\left( \frac{\varrho \left(\tau /4N\right)}{\sqrt{1-\theta}}\right)\right]^\gamma\log \frac{3N}{\epsilon}.\]

\textit{\textbf{Key words:}} second order cone optimization, interior point methods,  primal-dual methods, kernel functions.

\maketitle

\section{Introduction} \label{sec:intro}

In this work, we define a new type of \textit{second order cone optimization} (SOCO) problem, which is based on our definition of a {\em type-$k$ second order cone}. Second order conic programming involves optimizing a linear function over the Cartesian product of second order cones. A second order cone (see also Lorentz or ``ice cream'' cone) is defined as in \eqref{def_Lorentz}. 
\begin{align}
    \label{def_Lorentz}
    &\Lambda^n:=\left\{(x_1,x_2,...,x_n)\in \R^n : x_1^2 \geq \sum\limits_{i=2}^{n}x_i^2, x_1 \geq 0\right\}, & n\in\mathbb{Z}^+.
\end{align}

Considering $N$ second order cones and letting $\Lambda=\Lambda^{\left(1\right)} \times \Lambda^{\left(2\right)} \times \ldots \times \Lambda^{\left(N\right)}$ (where $\Lambda^{\left(i\right)}=\Lambda^{n_i}$), SOCO problems can then be defined as in \eqref{def_SOCO}.
\begin{subequations}
    \label{def_SOCO}
    \begin{align}
        \min~~ & c^T x \\
        \text{s.t.}~~ & Ax=b, \\
        ~~& x\in \Lambda.
    \end{align}
\end{subequations}
We are ready to introduce a type-$2$ second order cone in $\mathbb{R}^n$ as in \eqref{def_cone_type2}. 
\begin{align}
    \label{def_cone_type2}
    \Upsilon^n:=\left\{\mathbf{x} \in \R^n \ | \ (x_1+x_2)^2 \geq 2\sum\limits_{i=3}^{n}x_i^2, \ x_1 \geq x_2, \ x_1+x_2 \geq 0\right\}.
\end{align}
This leads, correspondingly, to the optimization problem we will primarily focus on in this work, the type-2 SOCO problem of \eqref{def_SOCO_type2}.
\begin{subequations}
    \label{def_SOCO_type2}
    \begin{align}
        \min~~ & c^T x \\
        \text{s.t.}~~ & Ax=b, \\
        ~~& x\in \Upsilon,
    \end{align}
\end{subequations}
where $\Upsilon$ is the cartesian product of $N$ type-$2$ second order cones (i.e., $\Upsilon=\Upsilon^{\left(1\right)}\times\Upsilon^{\left(2\right)}\times\ldots\times\Upsilon^{\left(N\right)}$, with $\Upsilon^{\left(i\right)}=\Upsilon^{n_i}$). 

Throughout the paper, we will make the assumption that matrix $A$ is of full rank. We will also partition the decision variable $x$ in $N$ components $x=\left(x^{\left(1\right)}, x^{\left(2\right)}, \ldots, x^{\left(N\right)}\right)$ with $x^{\left(j\right)}\in \Upsilon^{\left(j\right)}$, as well as the cost vector $c=\left(c^{\left(1\right)}, c^{\left(2\right)}, \ldots, c^{\left(N\right)}\right)$ with $c^{\left(j\right)}\in \mathbb{R}^{n_j}$. Finally, we will partition $A$ into $A=\left[A^{\left(1\right)}~A^{\left(2\right)}~\cdots~A^{\left(N\right)}~\right]$ with $A^{\left(j\right)}\in \mathbb{R}^{m\times n_j}$ and $b\in\mathbb{R}^m$.

\begin{remark}
\label{rem:def}
We remark here that the cone considered in this work is a generalized version of the traditional SOCO problem. Letting $u_1 = x_1+x_2, u_2 = x_1-x_2$, and $u_i=\sqrt{2} x_i, \forall i = 3,4,...,n$, then our type-2 cone can be cast as the following equivalent cone:
\begin{align}
    \label{def_cone_type2:eqv}
    \Upsilon^n:=\left\{\mathbf{u} \in \R^n \ | \ u_1^2 \geq  \sum\limits_{i=3}^{n}u_i^2, \ u_1 \geq 0, \ u_2 \geq 0\right\}.
\end{align}
This cone is different than in traditional SOCO, as this is defined in \eqref{def_Lorentz}. As a matter of fact SOCO and type-2 SOCO become equivalent if and only if $x_2$ and $u_2$ are fixed to 0 in \eqref{def_Lorentz} and \eqref{def_cone_type2:eqv} respectively, i.e, when $x_2 = u_2 = 0$ we have:
\begin{align*}
    \Lambda^n = \Upsilon^n:=\left\{\mathbf{x} \in \R^n \ | \ x_1^2 \geq  \sum\limits_{i=3}^{n}x_i^2, \ x_1 \geq 0 \right\}.
\end{align*}
\end{remark}

\begin{example}
To expand on the previous remark, consider the points $\mathbf{x} = \{4,3,3\}$ and $\mathbf{y} = \{5,-3,3\}$. Then, it is easy to check that $\mathbf{x} \in \Upsilon^3$, and $\mathbf{y} \in \Lambda^3$, but $\mathbf{x} \notin \Lambda^3,~\mathbf{y} \notin \Upsilon^3$, which implies that $\Upsilon^n \not\equiv \Lambda^n$.
\end{example}

Due to that, our type-2 SOCO (and its type-$k$ generalization, offered in Section \ref{sec:conclusion}) can be viewed as a traditional SOCO with a \textit{complicating variable} in the form of $x_2$. As will be discussed in Section \ref{sec:lit_review}, this is a problem that has yet to be studied in the literature. The objective of this work is to investigate the theoretical foundations of type-2 SOCO problems and generalize traditional SOCO in the presence of complicating variables. This definition can be generalized as in \eqref{typekSOCO} in the presence of more complicating variables:
\begin{align}
\label{typekSOCO}
\Omega^n:=\{\textbf{x}\in \R^n : [\sum_{i=1}^{k}x_i]^2 \geq \xi_k \sum\limits_{j=k+1}^{n}x_j^2, g_l(x_{1:k})\geq 0, \ x_r \geq 0, \ r,l \in (1,...,k)\}.
\end{align}

\begin{remark}
\label{rem:contr}
In this remark, we show that there is an algebraic transformation of a type-2 SOCO problem into a regular SOCO problem. Let us define $x_1-x_2 -z_1 = 0, \  1/\sqrt{2} x_1 + 1/\sqrt{2} x_2 - \bar{z}_1= 0, \ x_i = \bar{z}_{i-1}$ for $i = 3,4,...,n$. In addition, we define the following:
\begin{align*}
    & \bar{z} = [\bar{z}_1,...,\bar{z}_{n-1}]^T \in \R^{n-1}, \ \ e^T = [1,-1,0,...,0]^T \in \R^{n} \\
    & \acute{A} = \begin{bmatrix}
    \frac{1}{\sqrt{2}} & \frac{1}{\sqrt{2}} & 0 \hdots 0 \\
    0 & 0 & I 
\end{bmatrix} \in \R^{(n-1) \times n}, \ \ z = \begin{bmatrix}
    x \\
    z_1 \\
    \bar{z}
\end{bmatrix}, \ \bar{c} = \begin{bmatrix}
    c \\
    0 \\
    0
\end{bmatrix} \in \R^{2n} \\
    & \hat{A} = \begin{bmatrix}
    A & 0 & 0 \\
    e^T & -1 & 0 \\
    \acute{A} & 0 & -I
\end{bmatrix} \in \R^{(m+2n-1) \times 2n}, \ \ \hat{b} = \begin{bmatrix}
    b \\
    0 \\
    0
\end{bmatrix} \in \R^{m+2n-1}
\end{align*}
With the above definition, we can transform the type-2 SOCO problem of \eqref{def_SOCO_type2} into the following problem:
\begin{align}
\label{prob:x}
        \min~~ & \bar{c}^T z \nonumber \\
        \text{s.t.}~~ & \hat{A}z=\hat{b}, \\
        ~~& z \in \R^n_{+} \times \R_{+} \times \Lambda^{n-1} \nonumber
\end{align}
The above problem of \eqref{prob:x} can be viewed as a regular SOCO problem which can be solved using the interior method proposed in \cite{bai:2009}. However, using the regular interior point method for solving \eqref{prob:x} can have some disadvantages, which we describe next.

First, we note that the transformation changes the coefficient matrix from $A \in \R^{m \times n}$ to $\hat{A} \in \R^{(m+2n-1) \times 2n}$, the number of rows and columns of the coefficient matrix of the revised search direction equation (see system \eqref{direction_system}) increases by approximately $3n$. As the interior point method requires us to solve the search direction system of equations as fast as possible at each iteration, the bigger search direction system becomes more unstable (i.e., the probability of existence of the matrix inverse for the resulting coefficient matrix becomes lower, while even if the inverse does exist, it becomes harder to calculate). Furthermore, note that the iteration complexity of type-1 and type-2 SOCO is roughly same for the same matrix size $A$. Since the matrix size of the transformation problem \eqref{prob:x} increases from $m \times n$ to $m+2n-1 \times 2n$, the iteration complexity bound of type-2 problem is worse than the bound of type-1 problem when both problems are solved using the method proposed in Bai \textit{et al}. \cite{bai:2009}. These reasons are the main motivation in developing a special interior point method for type-2 SOCO problems.
\end{remark}

The outline of the paper is as follows. In Section \ref{sec:intro}, we have introduced the type-2 SOCO problem and how it differs to the regular SOCO problem. In Section \ref{sec:lit_review}, we discuss some recent literature with second order conic programming especially in relation to primal-dual methods and using kernel functions. Section \ref{sec:main} encompasses the main contributions of our work. In that section, we first provide the necessary fundamentals and then proceed to describe the finer details of the primal-dual interior point algorithm presented. We finish the section with a technical analysis on its theoretical performance, deriving an iteration bound. Throughout the paper and that section in particular, we present the similarities and differences of our type-2 second order conic program and the technical analysis to the one for a regular second order conic program. The paper concludes in Section \ref{sec:conclusion} with ideas for future investigations and the generalization to a type-$k$ SOCO problem.

The generic algorithm presented here and its analysis stem from the analysis performed in \cite{bai:2009} and \cite{bai:2004} for an interior point primal-dual method for linear optimization. Using some of their results and arguments, and deriving new theorems, we also obtain an iteration bound that is different (as explained in \ref{rem:def} about the uniqueness of our problem).

\section{Literature review}
\label{sec:lit_review}

In this section, we present some of the related literature on interior point primal-dual algorithms and how they have been applied for second order conic optimization problems. Solving SOCO problems has attracted significant interest and attention due to the wide range of applications in very different contexts (see \cite{lobo:1998}, \cite{alizadeh:2003}, \cite{kuo:2004}, \cite{tseng:2007}). The more general framework of \textit{semi-definite programming} SDP is studied as a generalization of SOCO, and hence any SDP approach can be also applied to solve SOCO problems. In order to efficiently implement approaches for SDP and SOCO problems for large-scale instances, scalable methods are needed. Recently, one such scalable method which is a primal-dual \textit{interior point method} has gained popularity for solving large-scale SOCO problems. As can be seen from the literature review in the rest of the section, employing primal-dual interior point algorithms for this type of problems is a very popular idea.

Adler and Alizadeh \cite{adler1995primal} were among the first to study a unified primal-dual approach for SDP and SOCO problems. In their work, they proposed a new search direction for SOCO problem, similar to the search direction defined for SDP. Subsequently after that, Ross \textit{et al}.\ \cite{roos1997theory} provided a brief theoretical overview of interior point methods as those are used in \textit{linear programming} in their book. An important theoretical benchmark was achieved by Nesterov \textit{et al}.\ \cite{nesterov:1997,nesterov:1998}, who showed that primal-dual interior point methods maintain their theoretical efficiency when the corresponding non-negativity constraints in linear programming are replaced by a convex cone. They went on to argue then that the distinguishing criterion is that the resulting cone must be self-dual and homogeneous.

In the early 2000s, in a series of research works, Schmieta \textit{et al}.\ \cite{alizadeh:2001,schmieta:2003} presented a novel way to transfer the Jordan algebra for SOCO into the well-known Clifford algebra in the domain of cones of matrices. In addition, using Jordan algebraic techniques, Faybusovich \cite{Faybusovich:1997,Faybusovich:2002} showed that primal-dual methods for semi-definite optimization can be extended for symmetric optimization, and analyzed a series of search directions for SOCO, including the Nesterov-Todd search direction proposed in \cite{nesterov:1997,nesterov:1998}. 

Continuing with interior point methods for linear optimization, Peng \textit{et al}.\ \cite{peng:2001} designed efficient algorithms employing a large update method. After that, in another series of papers Peng \textit{et al}.\ \cite{peng2:2002,peng3:2002} went ahead and proposed primal-dual interior point methods for both SDO and SOCO. Their work initiated research on the use of barrier functions for such types of problems, as they replaced the well-known logarithmic barrier by a self-regular barrier function, along with a proper modification of the search direction. A detailed theoretical overview of the self-regularity paradigm in the context of primal-dual interior point methods has been provided by Peng \textit{et al}.\ in their work \cite{peng:2002}.

In \cite{article}, Darvay proposed a method to derive a new class of search directions for linear optimization. They then used a similar method to derive search directions for a primal-dual scheme for solving self-dual linear optimization problems. Bai \textit{et al}.\ \cite{bai:2002} introduced a new, not self-regular barrier-type function based on a kernel function. They used this to develop an efficient, large-update primal-dual interior point method. A primal-dual interior point method was then designed by Andersen \textit{et al}.\ in their work \cite{andersen:2003} for solving conic quadratic optimization problems. Following their previous work, Bai \textit{et al}.\ \cite{bai:2004} then conducted a comparative analysis of kernel functions for primal- dual interior point methods in linear optimization. Using a simple kernel function, Wang \textit{et al}.\ also \cite{wang:2005} proposed a primal-dual interior point algorithm for semi-definite programming. 

More recently, Bai \textit{et al}.\ \cite{bai:2007} introduced a parametric kernel function for designing primal-dual interior point methods for SOCO problems. In 2008, Bai \textit{et al}.\ \cite{bai:2008} proposed a new class of polynomial interior point methods for solving the \textit{linear complementary problem}. In 2009, Bai \textit{et al}.\ proposed a generalized kernel function framework \cite{bai:2009} improving on their previous work using a parametric kernel function \cite{bai:2007} for solving SOCO problems. All of the above mentioned methods involve extensively using of a so-called kernel function (barrier function) and its unique properties.

 Based on that same idea, Wang \textit{et al}.\ developed a primal-dual interior point method to solve convex quadratic semi-definite optimization problems \cite{wang:2009} and SOCO with a Nesterov-Todd step \cite{wang2:2009}, respectively. Extending their previous work in \cite{wang:2009} and \cite{wang2:2009}, Wang \textit{et al}.\ \cite{WANG:2013329} developed full Nesterov-Todd step feasible interior point methods for convex quadratic symmetric cone optimization. The same idea was exploited by Kheirfam \textit{et al}.\ \cite{kheirfam:2014}: their work proposed a Nesterov-Todd step interior point method for solving a symmetric cone linear complementarity problem. Later, Cai \textit{et al}.\ also extended previous work on developing primal-dual interior point methods for solving convex quadratic optimization, based on a finite barrier function and convex quadratic optimization over a symmetric cone in their work in \cite{cai:2013} and \cite{Cai:2014}, respectively. Recently several researchers proposed infeasible interior point methods with full Nesterov-Todd step for solving SOCO problem (\cite{zangiabadi:2013}, \cite{Kheirfam:2018}). Some recent works suggest that developing efficient interior point methods for SOCO problems can be incorporated from a different general class of methods (i.e, smoothing Newton method \cite{TANG20111317}, predictor corrector method \cite{Zeng2011}, proximity function method \cite{fathi:2016}, local kernel function method \cite{lipu:2013}, among others).
 
 Our work is unique in the sense that it investigates a SOCO problem with a complicating variable in the form of $x_2$. Employing known results from the aforementioned literature, and through the use of a suitably defined (for our framework with a complicating variable) barrier function, we are able to show that our algorithm also achieves an iteration bound. We can now proceed to the main results of our work in Section \ref{sec:main}.

\section{A primal-dual interior point method for type-2 second order cone optimization problems}
\label{sec:main}

This is the main section of our work. We begin by offering some notation and fundamentals, and proceed to describe the details of how the algorithm works. The section concludes with a description and a technical analysis of the algorithm. 

\subsection{Fundamentals}

As is standard, we use $\mathbb{R}, \mathbb{R}_+$, and $\mathbb{R}_{++}$ to denote the set of real, non-negative real, and strictly positive real numbers, respectively. Matrices with real-valued entries on $m$ rows and $n$ columns will belong to set $\mathbb{R}^{m\times n}$, with $A_{ij}$ denoting the real-valued element in row $i$ and column $j$. $A^T$ will be used to denote the transpose of matrix $A$, with $\textbf{tr}(A)$, $\textbf{det}(A)$, and $\textbf{diag}(A)$ signaling the trace, determinant, and diagonal of matrix $A$. We will write that $A\in S^n$ if $A$ is an $n\times n$ symmetric matrix: $S^n_+$ and $S^n_{++}$ will be used to signal all positive semidefinite and definite symmetric matrices. $I_n$ will be used as the $n\times n$ identity matrix.

Furthermore, we will use vectors $\mathbf{1}=\left[1~1~\ldots~1\right]^T$ and $e_i$ as the standard $i$-th basis vector. A function $f:X\mapsto Y$ maps its domain, $dom(f)\subseteq X$, into set $Y$. As is customary, $\nabla f$ and $\nabla^2 f$ are used to represent the gradient and Hessian of $f$. Finally, $\langle x,y\rangle = x^Ty$ denotes the standard inner product and $\|x\| = \sqrt{\langle x, x \rangle}$ as the euclidean ($L_2$) norm.

Following the work and theory in \cite{bai:2009,bai:2004,peng:2002}, we proceed to define an operator $\diamond$, $\forall x,y \in \R^n$ as: 
\begin{multline}
\label{eq:5}
x\diamond y:= [x^Ty, \ x_2y_1+x_1y_2+x_{3:n}^Ty_{3:n}, \ x_3(y_1+y_2)+y_3(x_1+x_2), \ldots, \\
\ x_n(y_1+y_2)+y_n(x_1+x_2)].
\end{multline}
Equivalently to the work in \cite{bai:2009}, the $\diamond$ operator here also defines a Jordan algebra and forms a linear map, with $e$ serving as its unit vector. The matrix of the linear map, $R(x)$ is: 
\begin{align}
\label{eq:6}
R(x)=\begin{bmatrix}
    x_{1} & x_{2} & x_{3:n}^T \\
    x_{2} & x_{1} & x_{3:n}^T \\
    x_{3:n} & x_{3:n} & (x_1+x_2)I_{n-2} 
\end{bmatrix}.
\end{align}
$R(x)$ has four eigenvalues denoted by $\lambda_1, \lambda_2, \lambda_3, \lambda_4$. The eigenvalues and the corresponding eigenvectors are:
\begin{equation}
\begin{aligned}
	\label{eq:7}
	&\lambda_1(x)= x_1-x_2, && v_1=[\frac{1}{2}, -\frac{1}{2},\mathbf{0}]; \\
	&\lambda_2(x)= x_1+x_2-\sqrt{2}\|x_{3:n}\|, && v_2= [\frac{1}{4},\frac{1}{4},\frac{-x_{3:n}}{2\sqrt{2}\|x_{3:n}\|}]; \\
	&\lambda_3(x)= x_1+x_2, && v_3=[0, 0,\mathbf{y}]; \\
	&\lambda_4(x)= x_1+x_2+\sqrt{2}\|x_{3:n}\|, && v_4= [\frac{1}{4},\frac{1}{4},\frac{x_{3:n}}{2\sqrt{2}\|x_{3:n}\|}].
\end{aligned}
\end{equation}
In \eqref{eq:7}, for the eigenvector corresponding to $\lambda_3$, we assume that $y \in \R^{n-2}$ such that $x_{3:n}^Ty = 0$. There are two observations we can make now: (i) if $x\in\Upsilon^n$, then $\lambda_1(x)\geq 0$ and (ii) if $x\in\Upsilon^n_+$, then $\lambda_1(x)>0$. In addition, we can define a unique spectral decomposition for every vector $x\in\mathbb{R}^n$ as in \eqref{spectral}. 
\begin{align}
\label{spectral}
x=\lambda_1(x)v_1+\lambda_2(x)v_2+\lambda_4(x)v_4.
\end{align}

Based on \eqref{spectral}, we proceed to give a barrier function, $\Psi(x)$. First, for all $x\in\Upsilon^n$, let $\psi: \mathbb{R}_{++}\mapsto \mathbb{R}$ and define $\psi(x):=\psi(\lambda_1(x))v_1+\psi(\lambda_2(x))v_2+\psi(\lambda_4(x))v_4$. Now, define $\Psi(x)$ for all $x\in\Upsilon^n$ as: 
\begin{align}
	\label{barrier1}
	\Psi(x) = \textbf{Tr}\left(\psi(x)\right),
\end{align} 
where $\textbf{Tr}\left(\cdot\right)$ is used to denote the trace. For $x\in\Upsilon^n$ we use $\textbf{Tr}\left(x\right)=\lambda_1(x)+\frac{1}{2}\left(\lambda_2(x)+\lambda_4(x)\right)=2x_1.$ Hence, we can use this result in \eqref{barrier1} to obtain the (induced by kernel function $\psi\left(\cdot\right)$) barrier function shown in \eqref{barrier2}. 
\begin{align}
	\label{barrier2}
	\Psi(x) = \textbf{Tr}(\psi(x))=2(\psi(x))_1=\lambda_1(\psi(x))+\frac{1}{2}[\lambda_2(\psi(x))+\lambda_4(\psi(x))]
\end{align} 
Showing that $\Psi(x)$ is indeed a barrier function for the cone $\Upsilon^n$ follows closely the derivations in \cite{bai:2009} and is omitted. Before ending this subsection, we also define $\textbf{det}(x)=\frac{1}{2}(\lambda_1^2(x)+\lambda_2(x)\lambda_4(x))=x_1^2+x_2^2-||x_{3:n}||^2$, $\overline{\textbf{det}}(x)=\lambda_2(x)\lambda_4(x)$ $=(x_1+x_2)^2-2||x_{3:n}||^2$, and $\underline{\textbf{det}}(x)=\overline{det}(x)-det(x)=\frac{1}{2}(\lambda_2(x)\lambda_4(x) -\lambda_1^2(x))$ $=2x_1x_2-||x_{3:n}||^2$. Finally, because of $\textbf{Tr}(x)=\lambda_1(x)+\frac{1}{2}(\lambda_2(x)+\lambda_4(x))=2x_1$, we can also derive the identity in \eqref{useful_identity}. \begin{equation}\label{useful_identity}\begin{aligned} &\textbf{Tr}\left(x\diamond s\right)=2x^Ts, \\  &\textbf{Tr}\left(x\diamond x\right)=2x^Tx=2||x||^2. \\ \end{aligned}\end{equation}They are useful for some later, finer details.

\subsection{Algorithmic details}

In this subsection, we will assume (for simplicity of presentation) that we only have $N=1$ (we will later show the general case). There are three more steps before we are ready to present the algorithm:
\begin{enumerate}
    \itemsep0em
	\item Cone re-scaling.
	\item Central path construction.
	\item Search direction.
\end{enumerate}

Similarly to other works attempting to solve an original second order conic optimization problem, we also re-scale the space of the cone $\Upsilon$. For any $x_0, s_0\in\Upsilon_+$, there exists an automorphism $W(x_0, s_0)=W^{-1}(x_0,s_0)s_0$ of $\Upsilon$. With a slight misuse of notation, we write $W$ instead of $W(x_0,s_0)$. Then, for any $x,s\in\mathbb{R}^n$, $\overline{x}=Wx$ and $\overline{s}=W^{-1}s$ are known as the Nesterov-Todd scaling of $\mathbb{R}^n$ \cite{nesterov:1997,nesterov:1998}. 

Let $\Upsilon_+$ denote the interior of cone $\Upsilon$; further assume that both primal and dual problems have some point $x_0, s_0\in \Upsilon_+$ that satisfy primal and dual feasibility conditions. Using the self-dual embedding technique developed in \cite{luo:2000}, we can assume that $x_0=s_0=e$. The basic idea now behind using a primal-dual interior point method is to replace complementary slackness (the third optimality condition along with primal and dual feasibility) with the condition that $R(x)s=\mu e$, for $\mu>0$, leading to the optimality condition system of \eqref{opt_conditions}.
\begin{equation}
\label{opt_conditions}
\begin{aligned}
& Ax=b, & x\in\Upsilon,\\
&A^Ty+s=c, & s\in\Upsilon, \\
&R(x)s=\mu e.
\end{aligned}
\end{equation}
in lieu of the original third optimality condition of $R(x)s=0$. For every $\mu>0$ we obtain $\left(x(\mu), y(\mu), s(\mu)\right)$ as the solution of \eqref{opt_conditions}. The sets $\left\{x(\mu), \mu>0\right\}$ and $\left\{\left(y(\mu),s(\mu)\right), \mu>0\right\}$ give a central path in the primal and dual spaces. As $\mu\rightarrow 0$, the limits $\lim_{\mu\rightarrow 0} x(\mu)$ and $\lim_{\mu\rightarrow 0} \left(y(\mu), s(\mu)\right)$ satisfy $R(x)s=0$, and are respectively optimal solutions to the primal and dual problems. 

Finally, we proceed to describe the search direction. In accordance to \cite{bai:2009}, we linearize the system in \eqref{opt_conditions}, we use the same automorphism $W$ (shown earlier in this subsection), and we show that the same system of equations for defining the search direction holds, shown in \eqref{direction_system}, with $\overline{A}=\frac{AW^{-1}}{\sqrt{\mu}}$. 
\begin{equation}
\label{direction_system}
\begin{aligned}
&\overline{A}d_x=0 \\
&\overline{A}^T\Delta y+d_s=0 \\
&d_x+d_s=\psi^\prime(v).
\end{aligned}
\end{equation}

In the system, $d_x:=\frac{W\Delta x}{\sqrt{\mu}}, d_s:=\frac{W^{-1}\Delta s}{\sqrt{\mu}}$ are orthogonal vectors that are used to denote the scaled search direction, following the Nesterov-Todd scaling scheme discussed earlier in the subsection. To conclude that the directions obtained are obtained are valid whenever the current iterate $\left(x,y,s\right)$ is different than the optimal $\left(x(\mu), y(\mu),s(\mu)\right)$, we also need to show that $\psi^\prime(v)\neq 0$ concurrently holds. This leads to the following Theorem \ref{thm:1}.

\begin{theorem}[Adapted from Lemma 2.10 in \cite{bai:2009}]
\label{thm:1}

Let $v=\frac{Wx}{\sqrt{\mu}}=\frac{W^{-1}s}{\sqrt{\mu}}$. Then, $\psi^\prime(v)=0$ if and only if $x\diamond s=\mu e$. 

\end{theorem}

\begin{proof}
We start by showing that $\psi'(v)=0 $ if and only if $v=e$. By definition of a kernel function, it is strictly convex and minimal at $t=1$. This, in turn, is true if and only if $\lambda_1(x)=\lambda_2(x)=\lambda_4(x)=1$, (i.e., if and only if $x=e$). Hence, we have that:
\begin{align}
\label{eq:14}
\Psi(x)=0 \Leftrightarrow \psi(x)=0\Leftrightarrow \psi^{\prime}(x)=0 \Leftrightarrow x=e
\end{align}
Since $v=\frac{Wx}{\sqrt{\mu}}=\frac{W^{-1}s}{\sqrt{\mu}}$ (by assumption), we have $x=\sqrt{\mu}W^{-1}e$ and $s=\sqrt{\mu}We$. From Theorem \ref{th:22} (Appendix), $W$ can be written as: 
\[W=\sqrt{\lambda}W_a, \quad \text{with} \quad W_a=\begin{bmatrix}
    a_{1} & a_{2} & \bar{a}^T \\
    a_{2} & a_{1} & \bar{a}^T \\
    \bar{a} & \bar{a} & I_{n-2}+\frac{2 \bar{a}\bar{a}^T}{1+a_1+a_2} 
\end{bmatrix},\]
\noindent where $a=(a_1,a_2;\bar{a})$ is a vector such that $\textbf{det}(a)= \overline{\textbf{det}}(a)=1$ and $\lambda > 0, \lambda\neq 1$. It follows that $W^{-1}=\frac{1}{\sqrt{\lambda}}W_{Qa}$ with $Q=\text{diag}(1,1,-1,...,-1)\in \R^{n\times n}$ and $Qa=(a_1,a_2;-\bar{a})$. It follows that $v=e$ holds if and only if:
\[x=\sqrt{\mu}W^{-1}e=\frac{\sqrt{\mu}}{\sqrt{\lambda}}W_{Qa}e= \frac{\sqrt{\mu}}{\sqrt{\lambda}}\begin{bmatrix}
    a_{1}\\
    a_{2}\\
    -\bar{a} 
\end{bmatrix} \quad \text{and} \quad s=\sqrt{\mu}We= \sqrt{\mu}\sqrt{\lambda}\begin{bmatrix}
    a_{1}\\
    a_{2}\\
    \bar{a} 
\end{bmatrix}.\]
Now, from our definition of the Jordan product ($\diamond$), we have that:
\[x \diamond s=\frac{\sqrt{\mu}}{\sqrt{\lambda}}\sqrt{\mu}\sqrt{\lambda}\begin{bmatrix}
    a_{1}\\
    a_{2}\\
    -\bar{a} 
\end{bmatrix} \diamond \begin{bmatrix}
    a_{1}\\
    a_{2}\\
    \bar{a} 
\end{bmatrix}=\mu \begin{bmatrix}
    a_{1}^2+a_2^2-\|\bar{a}\|^2\\
    2a_1a_{2}-\|\bar{a}\|^2\\
    -(a_1+a_2)\bar{a}+(a_1+a_2)\bar{a} 
\end{bmatrix}=\mu e.\]
This concludes the first direction of the proof. Now conversely, if $x \diamond s =\mu e$ then there must exist some $\lambda > 0, \lambda\neq 1$ and some vector $a$ such that $x$ and $s$ has the above form. This shows the other direction and concludes the proof.
\end{proof}

We now go back and show the general case, where $N>1$ and hence $\Upsilon$ is defined as $\Upsilon^{\left(1\right)}\times\Upsilon^{\left(2\right)}\times\ldots\times\Upsilon^{\left(N\right)}$, with $\Upsilon^{\left(i\right)}=\Upsilon^{n_i}$. Once again, following in the footsteps of \cite{bai:2009}, we can generalize and show that the discussion in this subsection is extended to $N>1$. 

First, partition $x$ into $N$ components $\left(x^1, \ldots, x^N\right)$. Also, generalize the $\diamond$ operator as in $x\diamond s=\left(x^1\diamond s^1, \ldots, x^N\diamond s^N\right)$. Furthermore, if $e^j\in \Upsilon^j$ is a unit element for the Jordan algebra defined by $\diamond$ for $\Upsilon^j$, then $e=\left(e^1, \ldots, e^N\right)$ is the unit element for the Jordan algebra of the generalized operator $\diamond$. 

We also obtain a Nesterov-Todd scaling mechanism for $N>1$, using $W^j$ as the automorphism (see $W$ earlier) for cone $\Upsilon^j$. The only difference in our mechanism (compared to the mechanism of \cite{bai:2009}) is that here we have: \[\lambda_j=\frac{\lambda_1(s^j)}{\lambda_1(x^j)}=\frac{s_1^j-s_2^j}{x_1^j-x_2^j}.\] The rest of the generalization process for the traditional cone second order cone $\Lambda^n$, presented in \eqref{def_Lorentz}, also holds here using $W=\textbf{diag}\left(W^1, \ldots, W^N\right)$ and hence is omitted. The interested reader is referred to \cite{bai:2009} for details of the generalization, keeping in mind the differences in our type-2 second order cone, defined in \eqref{def_cone_type2}. The search directions (back in the original space after the rescaling) are then as in \eqref{search_directions}.
\begin{equation}
\label{search_directions}
\begin{aligned}
&\Delta x=\sqrt{\mu}W^{-1}d_x, & \Delta s=\sqrt{\mu}W d_s.
\end{aligned}
\end{equation}
Finally, we use the same generalization (for $N>1$) as in \cite{bai:2009} for $\Psi(\cdot), \psi(\cdot)$:
\begin{align}
\label{eq:25}
\psi(v)=\left(\psi(v^1),\psi(v^2),...,\psi(v^N)\right), \quad \Psi(v)= \sum\limits_{j=1}^{N}\Psi(v^j)
\end{align}

\subsection{The algorithm} 

The generic algorithm (presented in Algorithm \ref{IPM_alg}) is well-studied in the literature (see, e.g., \cite{bai:2009,bai:2004,peng:2002}). Moving along the search direction with step size $\alpha$ measured by searching rule we construct new iterates $\left(x^+, y^+, s^+\right)$ as shown in \eqref{eq:26}.
\begin{equation}
\label{eq:26}
\begin{aligned}
& x^+= x+\alpha \Delta x, \\
& y^+= y+\alpha \Delta y, \\
& s^+= s+\alpha \Delta s. 
\end{aligned}
\end{equation}

\begin{algorithm}[htpb]
    \SetKwInOut{Input}{Input}
    \SetKwInOut{Output}{Output}

    \underline{function IPM} $(\tau,\epsilon,\theta,\alpha)$\;
    \Input{Threshold $\tau>1$, accuracy $\epsilon>0$, update $\theta\in\left(0,1\right)$, step size $\alpha$}
    \Output{A solution to the primal and dual problems $\left(x,y,s\right)$}
    $x \leftarrow e$, $s \leftarrow e$, $\mu \leftarrow 1$\;
    \While{$3N\mu\geq \epsilon$}{
    	 $\mu\leftarrow \left(1-\theta\right)\mu$\;
    	 \While{$\Phi(x,s,\mu)>\tau$}{
    	    $x \leftarrow x+\alpha\Delta x$\;
            $y \leftarrow y+\alpha\Delta y$\;
            $s \leftarrow s+\alpha\Delta s$\;}
  }
  	\Return $\left(x,y,s\right)$
	\caption{Primal-Dual Algorithm for Type-2 SOCO. \label{IPM_alg}}
\end{algorithm}

Here, we adapt the generic algorithm for our version of the type-2 SOCO problem we are trying to solve. Combining all the steps we can construct Algorithm \ref{IPM_alg}. We observe that the accuracy of the algorithm is measured by the distance between the optimal solutions $\left(x(\mu),y(\mu),s(\mu)\right)$ to the returned solution $\left(x,y,s\right)$, which in turn is described by the barrier function, $\Phi(x,s,\mu):=\Psi(v)$. If the value of $\Phi(x,s,\mu)$ is less than or equal to some threshold $\tau$ (see line 5 of Algorithm \ref{IPM_alg}), then we can update the $\mu$ value, which is set to decrease from iteration to iteration according to some update parameter $\theta\in\left(0,1\right)$. If the value of $\Psi(x,s,\mu)$ is greater than the threshold $\tau$, then the current iterate $(x,y,s)$ is updated according to directions $\Delta x, \Delta y, \Delta s$ and the step size $\alpha$ (lines 6--8). This process continues until the optimality gap (which is equal to $3N\mu$) is less than some accuracy $\epsilon>0$ (see line 3).

\subsection{Technical analysis}

In this subsection, we prove the necessary convergence analysis of the proposed algorithm along with the theoretical iteration bound for the newly defined type-2 SOCO. In addition to that, we also state some generic lemmas for analyzing the performance of the original SOCO problem that are well studied in the literature (see \cite{bai:2009}, \cite{bai:2004}, \cite{peng:2002}, among others). 

We begin the subsection with the conditions that define eligible kernel functions. For more details on the conditions and a much deeper discussion, we refer the interested readers to one of the above references or  \cite{lipu:2013}. A kernel function is referred to as eligible if it satisfies:

\begingroup
\allowdisplaybreaks
\begin{align}
&t \psi^{\prime\prime}(t)+\psi^{\prime}(t)> 0, t<1 \label{eq:27}\\
&t \psi^{\prime\prime}(t)-\psi^{\prime}(t)> 0, t>1 \label{eq:28}\\
& \psi^{\prime\prime}(t)< 0, t>0 \label{eq:29} \\
& 2\psi^{\prime\prime}(t)^2-\psi^{\prime}(t)\psi^{\prime}(t)> 0, t<1 \label{eq:30} \\
& \psi^{\prime\prime}(t)\psi^{\prime}(\beta t)-\beta \psi^{\prime}(t)\psi^{\prime\prime}(\beta t)(t)> 0, t>1, \beta >1 \label{eq:31}
\end{align}
\endgroup

\begin{lemma}
\label{lem:2}
Letting $v=\frac{Wx}{\sqrt{\mu}}=\frac{W^{-1}s}{\sqrt{\mu}}$ as defined earlier, the following hold:

\begin{equation}
\begin{aligned}
&\mu^2\overline{\textbf{det}}(v \diamond v)=\overline{\textbf{det}}(x)\overline{\textbf{det}}(s), \\ 
&\mu \textbf{Tr}(v \diamond v)=\textbf{Tr}(x)\textbf{Tr}(s), \\
&\mu \lambda_1(v \diamond v)= \lambda_1(x) \lambda_1(s).
\end{aligned}
\end{equation}

\end{lemma}

\begin{proof}
Using the definition of $v$ and considering Lemmas \ref{lem:24}, \ref{lem:25} and \ref{lem:26} from the Appendix, we can easily get all identities.
\end{proof}

\begin{theorem}
\label{th:3}
Let $x,s,v \in \Upsilon_+$, which as a reminder is the interior of cone $\Upsilon$, and assume they satisfy the conditions in \eqref{th:3:conditions}.

\begin{equation}
\label{th:3:conditions}
\begin{aligned}
&\overline{\textbf{det}}(v \diamond v)=\overline{\textbf{det}}(x)\overline{\textbf{det}}(s), \\ 
&\textbf{Tr}(v \diamond v)=\textbf{Tr}(x \diamond s), \\  
&\lambda^2_1(v)=\lambda_1\left(x \diamond s\right)=\lambda_1(x)\lambda_1(s).
\end{aligned}
\end{equation}

\noindent Then, the inequality in \eqref{eq:32} also holds.
\begin{align}
\label{eq:32}
\Psi(v)\leq \frac{1}{2} \Psi(x)+\frac{1}{2}\Psi(s)
\end{align}
\end{theorem}

For proving Theorem \ref{th:3}, we need the following Lemma \ref{lemma_peng} from \cite{peng:2002}, which is provided without proof. 

\begin{lemma}[\cite{peng:2002}]
\label{lemma_peng}
For every $t_1, t_2 > 0$ we have
\begin{align}
\label{eq:33}
\psi(\sqrt{t_1 t_2})\leq \frac{1}{2} \psi(t_1)+\frac{1}{2}\psi(t_2).
\end{align}
\end{lemma}

\begin{proof}
Using Lemma \ref{lemma_peng}, we can prove the following bounds: 
\[\lambda_2(v)\geq \sqrt{\lambda_2(x)\lambda_2(s)} \quad \text{and} \quad \lambda_4(v)\leq \sqrt{\lambda_4(x)\lambda_4(s)}.\] Using the definitions and given conditions, we have (for any $v$) that the following hold: \[\lambda^2_2(v)+\lambda^2_4(v)= 4 \|v\|^2-2 \lambda_1^2(v)=2 \textbf{Tr}(v \diamond v)-2 \lambda_1^2(v)=2\textbf{Tr}(x \diamond s)-2 \lambda_1(x)\lambda_1(s)\]In addition, from Lemma \ref{lem:16} (in the Appendix), we have
\[ 2 \textbf{Tr}(x\diamond s)-2\lambda_1(x)\lambda_1(s)\leq \lambda_2(x)\lambda_2(s)+\lambda_4(x)\lambda_4(s)\]
This in turn gives 
\begin{align*}
\left[\lambda_2(v)+\lambda_4(v)\right]^2&=\lambda^2_2(v)+\lambda^2_4(v)+2\lambda_2(v)\lambda_4(v)\\
&=2\textbf{Tr}(x\diamond s)-2\lambda_1(x)\lambda_1(s)+2\sqrt{\lambda_2(x)\lambda_2(s)}\sqrt{\lambda_4(x)\lambda_4(s)}\\
& \leq \lambda_2(x)\lambda_2(s)+\lambda_4(x)\lambda_4(s)+2\sqrt{\lambda_2(x)\lambda_2(s)}\sqrt{\lambda_4(x)\lambda_4(s)}\\
&= \left[\sqrt{\lambda_2(x)\lambda_2(s)}+\sqrt{\lambda_4(x)\lambda_4(s)}\right]^2
\end{align*} 
Since, both sides are positive, we can take the square roots and get:
\begin{align}
\label{eq:34}
\lambda_2(v)+\lambda_4(v)\leq \sqrt{\lambda_2(x)\lambda_2(s)}+\sqrt{\lambda_4(x)\lambda_4(s)}
\end{align}
Similarly, we have:
\begin{align*}
\left[\lambda_4(v)-\lambda_2(v)\right]^2&=\lambda^2_2(v)+\lambda^2_4(v)-2\lambda_2(v)\lambda_4(v)\\
&=2\textbf{Tr}(x\diamond s)-2\lambda_1(x)\lambda_1(s)-2\sqrt{\lambda_2(x)\lambda_2(s)}\sqrt{\lambda_4(x)\lambda_4(s)}\\
& \leq \lambda_2(x)\lambda_2(s)+\lambda_4(x)\lambda_4(s)-2\sqrt{\lambda_2(x)\lambda_2(s)}\sqrt{\lambda_4(x)\lambda_4(s)}\\
&= \left[\sqrt{\lambda_4(x)\lambda_4(s)}-\sqrt{\lambda_2(x)\lambda_2(s)}\right]^2
\end{align*}
As $\lambda_4(x)\lambda_4(s)\geq \lambda_2(x)\lambda_2(s)$, taking the square roots again, we get:
\begin{align}
\label{eq:35}
\lambda_4(v)-\lambda_2(v)\leq \sqrt{\lambda_4(x)\lambda_4(s)}-\sqrt{\lambda_2(x)\lambda_2(s)}.
\end{align}
Adding \eqref{eq:34} to \eqref{eq:35}, we have the inequality in \eqref{eq:36}.
\begin{align}
\label{eq:36}
\lambda_4(v)\leq \sqrt{\lambda_4(x)\lambda_4(s)}.
\end{align}
Considering the upper bound in \eqref{eq:36} with the fact that
\[\lambda_2(v)\lambda_4(v)=\sqrt{\lambda_2(x)\lambda_2(s)}\sqrt{\lambda_4(x)\lambda_4(s)},\]
we can calculate the lower bound as
\[\lambda_2(v)\geq \sqrt{\lambda_2(x)\lambda_2(s)}.\]
Therefore, we have shown both upper and lower bounds. Now, as shown in \cite{peng:2002}, there exists a constant $r \in [\frac{1}{2},1]$ such that the following relations hold:
\begin{align}
&\lambda_2(v)=\lambda^{\frac{r}{2}}_2(x)\lambda^{\frac{r}{2}}_2(s)\lambda^{\frac{1-r}{2}}_4(x)\lambda^{\frac{1-r}{2}}_4(s) \label{eq:37} \\
& \lambda_4(v)=\lambda^{\frac{1-r}{2}}_2(x)\lambda^{\frac{1-r}{2}}_2(s)\lambda^{\frac{r}{2}}_4(x)\lambda^{\frac{r}{2}}_4(s) \label{eq:38}
\end{align}
Finally, using the definition in \eqref{barrier2} and the identities from equations \eqref{eq:37} and \eqref{eq:38} we get that
\begin{align*}
\Psi(v)& = \psi(\lambda_1(v))+\frac{1}{2}\left[\psi(\lambda_2(v))+\psi(\lambda_4(v))\right]\\
&\leq  \psi\left(\sqrt{\lambda_1(x)\lambda_1(s)}\right)+\frac{1}{2}  \psi\left(\sqrt{\lambda_2(x)\lambda_2(s)}\right)+\frac{1}{2} \psi\left(\sqrt{\lambda_4(x)\lambda_4(s)}\right).
\end{align*}
Since all of $\lambda_1(x),\lambda_1(s),\lambda_2(x),\lambda_2(s),\lambda_4(x),\lambda_4(s) > 0$, using \eqref{eq:33} we have
\begin{align*}
& \Psi(v) \leq  \psi\left(\sqrt{\lambda_1(x)\lambda_1(s)}\right)+\frac{1}{2}  \psi\left(\sqrt{\lambda_2(x)\lambda_2(s)}\right)+\frac{1}{2} \psi\left(\sqrt{\lambda_4(x)\lambda_4(s)}\right)\\
& \leq \frac{1}{2}[ \psi\left(\lambda_1(x)\right)+\psi\left(\lambda_1(s)\right)]+\frac{1}{4}[ \psi\left(\lambda_2(x)\right)+\psi\left(\lambda_2(s)\right)+\psi\left(\lambda_4(x)\right)+ \psi\left(\lambda_4(s)\right)]\\
&= \frac{1}{2} \Psi(x)+\frac{1}{2} \Psi(s)
\end{align*}
This finishes the proof.
\end{proof}

Now for any general $v$, using \eqref{eq:14} we have
\begin{align}
\label{eq:39}
\Psi(v)=0 \Leftrightarrow \psi(v)=0\Leftrightarrow \psi^{\prime}(v)=0 \Leftrightarrow v=e
\end{align}

In the remainder of the section, we will use a norm-based proximity function, $\delta(v)$, defined as in \eqref{eq:40}. 
\begin{equation}
\begin{aligned}
\label{eq:40}
\delta(v)&= \frac{1}{\sqrt{2}}\sqrt{\sum\limits_{i=1}^{N}\|\psi^{\prime}(v^j)\|^2} \\
&= \frac{1}{2 \sqrt{2}}\sqrt{\sum_{j=1}^{N} \left[2\psi^{\prime}(\lambda_1(v^j))^2+\psi^{\prime}(\lambda_2(v^j))^2+\psi^{\prime}(\lambda_4(v^j))^2\right]}.
\end{aligned}
\end{equation}
Notice that, from \eqref{eq:40} and since $\delta(v)\geq 0$ we have that
$\delta(v)=0 \iff \Psi(v)=0$. Hence, after each iteration we will have a primal-dual pair iterate in $(x,s)$ which can be updated using \eqref{eq:26}.

\begin{theorem}
\label{th:new1}
Let $\alpha$ be a strictly feasible step size, i.e., $\alpha$ is such that $(x+\alpha \Delta x, s+\alpha \Delta s)\in \Upsilon_+$, and $v$ defined as earlier. Then, the $j$-th iterate $v^j$ satisfies the following: 
\begin{align*}
&\overline{\textbf{det}}[(v_+^j)^2]= \overline{\textbf{det}}\left(v^j+\alpha d_x^j\right)\overline{\textbf{det}}\left(v^j+\alpha d_s^j\right), \\
&  \textbf{Tr}[\left(v^j_+\right)^2 ]=\textbf{Tr}[\left(v^j+\alpha d_x^j\right) \diamond \left(v^j+\alpha d_s^j\right)], \\
& \lambda_1[\left(v^j_+\right)^2]=\lambda_1[\left(v^j+\alpha d_x^j\right) \diamond \left(v^j+\alpha d_s^j\right)], \\
& \Psi(v_+)\leq \frac{1}{2} \Psi\left(v+\alpha d_x\right)+\frac{1}{2}\Psi\left(v+\alpha d_s\right).
\end{align*}
\end{theorem}

\begin{proof}
In other words, feasibility of $\alpha$ implies that at each iteration we have that $(x^j+\alpha \Delta x^j ,s^j+\alpha \Delta s^j ) \in \Upsilon_{+}^j$. Let us denote $W^j$ as the automorphism of $\Upsilon^j$ that satisfies $W^j x^j= (W^j)^{-1}s^j , v^j=W^j x^j/\sqrt{\mu}$. Then the new $j$- th iterate must satisfy:
\begin{align*}
&W^j\left(x^j+\alpha \Delta x^j\right)= \sqrt{\mu}\left(v^j+\alpha d_x^j\right), \\
&(W^j)^{-1}\left(s^j+\alpha \Delta s^j\right)= \sqrt{\mu}\left(v^j+\alpha d_s^j\right).
\end{align*}
Recalling Lemma \ref{lem:2} with $\bar{x}=\sqrt{\mu}\left(v^j+\alpha d_x^j\right)$ and $\bar{s}=\sqrt{\mu}\left(v^j+\alpha d_s^j\right)$ we get
\begin{align*}
&\mu^2\overline{\textbf{det}}\left(v^j+\alpha d_x^j\right)\overline{\textbf{det}}\left(v^j+\alpha d_s^j\right)= \overline{\textbf{det}}\left(x^j+\alpha \Delta x^j\right)\overline{\textbf{det}}\left(s^j+\alpha \Delta s^j\right), \\
& \mu \textbf{Tr}[\left(v^j+\alpha d_x^j\right) \diamond \left(v^j+\alpha d_s^j\right) ]=\textbf{Tr}[(\left(x^j+\alpha \Delta x^j\right) \diamond \left(s^j+\alpha \Delta s^j\right) ], \\
& \mu^2 \lambda_1[\left(v^j+\alpha d_x^j\right) \diamond \left(v^j+\alpha d_s^j\right) ]=\lambda_1\left(x^j+\alpha \Delta x^j\right) \lambda_1 \left(s^j+\alpha \Delta s^j\right).
\end{align*}
Similarly, if $W_+^j$ is the automorphism that satisfies $W_+^j x^{+j}=(W_+^j)^{-1}s^{+j}$ and $v^{+j}= \frac{W_+^{j}x^{+j}}{\sqrt{\mu}}$, then recalling Lemma \ref{lem:2} again we have
\begin{align*}
&\mu^2\overline{\textbf{det}}[(v_+^j)^2]= \overline{\textbf{det}}\left(x^j+\alpha \Delta x^j\right)\overline{\textbf{det}}\left(s^j+\alpha \Delta s^j\right), \\
& \mu \textbf{Tr}[\left(v^j_+\right)^2]=\textbf{Tr}[\left(x^j+\alpha \Delta x^j\right) \diamond \left(s^j+\alpha \Delta s^j\right) ], \\
& \mu^2 \lambda_1[\left(v^j_+\right)^2]=\lambda_1\left(x^j+\alpha \Delta x^j\right) \lambda_1 \left(s^j+\alpha \Delta s^j\right).
\end{align*}
Therefore, for every $j$, we can conclude the following:
\begin{align*}
&\overline{\textbf{det}}[(v_+^j)^2]= \overline{\textbf{det}}\left(v^j+\alpha d_x^j\right)\overline{\textbf{det}}\left(v^j+\alpha d_s^j\right), \\
&  \textbf{Tr}[\left(v^j_+\right)^2 ]=\textbf{Tr}[\left(v^j+\alpha d_x^j\right) \diamond \left(v^j+\alpha d_s^j\right)], \\
& \lambda_1[\left(v^j_+\right)^2]=\lambda_1[\left(v^j+\alpha d_x^j\right) \diamond \left(v^j+\alpha d_s^j\right)].
\end{align*}
This proves the first three identities of Theorem \ref{th:new1}. For the last identity, we note that for every $j$ the following inequality holds due to Theorem \ref{th:3}: 
\[\Psi(v_+^j)\leq \frac{1}{2} \Psi\left(v^j+\alpha d_x^j\right)+\frac{1}{2}\Psi\left(v^j+\alpha d_s^j\right)\]
To finish the proof, we need only sum over all $1 \leq j \leq N$ and get that
\[\Psi(v_+)\leq \frac{1}{2} \Psi\left(v+\alpha d_x\right)+\frac{1}{2}\Psi\left(v+\alpha d_s\right),\]which completes the proof.
\end{proof}

Continuing with the technical analysis, let us denote the decrease in $\Psi(v)$ during an inner iteration (see lines 6--8 in Algorithm \ref{IPM_alg}) as: $f(\alpha):= \Psi(v_+)-\Psi(v)$. Then, one has the following identity:
\[f(\alpha) \leq f_1(\alpha):=  \frac{1}{2} \Psi\left(v+\alpha d_x\right)+\frac{1}{2}\Psi\left(v+\alpha d_s\right)- \Psi(v).\]
We can easily verify $f(0)=f_1(0)=0$. The idea of using such a decrease function is commonly used throughout the SOCO literature. For example, analyzing regular SOCO problems, various researchers have exploited different properties of $f(\alpha)$ (see, e.g., \cite{bai:2009} and \cite{peng:2002}). We derive an upper bound of $f_1(\alpha)$, which is a convex function, which is different than the one derived in \cite{bai:2009} or \cite{peng:2002}. At first, we calculate the 1st and 2nd derivatives of $f_1(\alpha)$ in terms of the Jordan product defined earlier as in \eqref{eq:41} and \eqref{eq:42}.
\begin{align}
& f_1^{\prime}(\alpha)= \frac{1}{2} \textbf{Tr}[\psi^{\prime}(v+\alpha d_x) \diamond d_x+ \psi^{\prime}(v+\alpha d_s) \diamond d_s] \label{eq:41}, \\
& f_1^{\prime\prime}(\alpha)= \frac{1}{2} \textbf{Tr}[(d_x\circ d_x) \diamond \psi^{\prime\prime}(v+\alpha d_x) + (d_s \circ d_s) \diamond \psi^{\prime\prime}(v+\alpha d_s)]. \label{eq:42} 
\end{align}
Replacing $\alpha=0$, we have \eqref{eq:43}.
\begin{align}
\label{eq:43}
 f_1^{\prime}(0)= \frac{1}{2} \textbf{Tr}\left(\psi^{\prime}(v) \diamond (d_x+d_s)\right)= -\frac{1}{2} \textbf{Tr}\left(\psi^{\prime}(v) \diamond \psi^{\prime}(v)\right)= - 2 \delta(v)^2
\end{align}
Now, let us define the following:
\[\lambda_2(v)= \text{min} \left\{\lambda_2(v^j): 1 \leq j \leq N\right\} \ \ \text{and} \ \ \lambda_4(v)= \text{max} \left\{\lambda_4(v^j): 1 \leq j \leq N\right\}.\]
With our definitions in hand, we embark to prove a series of modified versions of the original theorems given in \cite{bai:2009,bai:2004,peng:2002} suitable for our framework of type-2 SOCO. 

\begin{theorem}
\label{th:4}
If $\psi^{\prime\prime}(t)$ is monotonically decreasing in $t$ then
\[f_1^{\prime\prime}(\alpha) \leq 2 \delta(v)^2 \psi^{\prime\prime}\left(\lambda_2(v)-2 \sqrt{2} \alpha \delta(v)\right)\].
\end{theorem}

\begin{proof}
Since $d_x, d_s$ are orthogonal and $d_x+d_s = - \psi^{\prime}(v)$ we have
\[\|d_x+d_s\|^2=\|d_x\|^2+\|d_s\|^2=2 \delta(v)^2. \]
This gives us that $\|d_x\| \leq \sqrt{2}\delta(v),\|d_s\| \leq \sqrt{2}\delta(v)$. Now, employing Lemma \ref{lem:14} we have that
\begin{align*}
& \lambda_4 \left(v+ \alpha d_x \right) \geq \lambda_2 \left(v+ \alpha d_x \right)\geq \lambda_2(v)-2 \sqrt{2} \alpha \delta(v), \\
& \lambda_4 \left(v+ \alpha d_s \right) \geq \lambda_2 \left(v+ \alpha d_s \right)\geq \lambda_2(v)-2 \sqrt{2} \alpha \delta(v).
\end{align*}
Recalling the properties of the kernel function, shown in \eqref{eq:29}, we know that $\psi^{\prime\prime}(t)$ is positive and monotonically decreasing, which in turn gives us
\begin{align*}
& 0 < \psi^{\prime\prime}\left(\lambda_4(v+\alpha d_x)\right) \leq \psi^{\prime\prime}\left(\lambda_2(v)-2 \sqrt{2} \alpha \delta(v)\right), \\
& 0 < \psi^{\prime\prime}\left(\lambda_4(v+\alpha d_s)\right) \leq \psi^{\prime\prime}\left(\lambda_2(v)-2 \sqrt{2} \alpha \delta(v)\right).
\end{align*}
Considering Corollary \ref{cor:16} (from the Appendix), and using that $\textbf{Tr}(x \diamond x)= 2 \|x\|^2$ for any $x$ from \eqref{useful_identity}, we have
\begin{align}
\nonumber
\textbf{Tr}\left((d_x\diamond d_x) \diamond \psi^{\prime\prime}(v+\alpha d_x)\right)& \leq \psi^{\prime\prime}\left(\lambda_2(v)-2 \sqrt{2} \alpha \delta(v)\right) \textbf{Tr}(d_x \diamond d_x),\\
\nonumber
&=2\psi^{\prime\prime}\left(\lambda_2(v)-2 \sqrt{2} \alpha \delta(v)\right)\|d_x\|^2, \\
\nonumber
\textbf{Tr}\left((d_s\diamond d_s) \diamond \psi^{\prime\prime}(v+\alpha d_s)\right) & \leq \psi^{\prime\prime}\left(\lambda_2(v)-2 \sqrt{2} \alpha \delta(v)\right) \textbf{Tr}(d_s \diamond d_s),\\
\label{final_inequality_th:4}
&=2\psi^{\prime\prime}\left(\lambda_2(v)-2 \sqrt{2} \alpha \delta(v)\right)\|d_s\|^2.
\end{align}
Substituting the above upper bound from \eqref{final_inequality_th:4} in equation \eqref{eq:42} we finally have:
\begin{align*}f_1^{\prime\prime}(\alpha) &\leq \psi^{\prime\prime}\left(\lambda_2(v)-2 \sqrt{2} \alpha \delta(v)\right) (\|d_x\|^2+\|d_s\|^2)=\\&=2 \delta(v)^2\psi^{\prime\prime}\left(\lambda_2(v)-2 \sqrt{2} \alpha \delta(v)\right),\end{align*}
which proves the theorem.
\end{proof}

\begin{remark}
Theorem \ref{th:4} and its proof for our case allows us to use same materials proved in \cite{bai:2004} and \cite{bai:2009}. Note that the inequality in Theorem \ref{th:4} is different from the inequality in Lemma 4.1 of \cite{bai:2004} or Lemma 3.3 in \cite{bai:2009} by a factor $\sqrt{2}$. One can surmise that this $\sqrt{2}$ factor does not affect most of the proofs. For this reason, we now write some lemmas without proof and simply refer the interested reader to the corresponding results in the works of \cite{bai:2004}, \cite{peng:2002}, and \cite{bai:2009}.
\end{remark}

Now, theorems \ref{th:5} through \ref{th:10} discuss the decrease of the barrier function in every inner iteration of Algorithm \ref{IPM_alg}, while Lemma \ref{lem:11} and Corollary \ref{cor:11} every outer iteration of Algorithm \ref{IPM_alg}.

\begin{theorem}[Adapted from Lemma 4.2 in \cite{bai:2004}]
\label{th:5}
If $-\psi^{\prime}(\lambda_2(v)-2 \sqrt{2} \alpha \delta(v))$ $+ \psi^{\prime\prime}(\lambda_2(v))\leq 2 \sqrt{2}\delta(v)$, then $f_1^{\prime}(\alpha)\leq 0$.
\end{theorem}
\begin{proof}
From the definition of $f_1^\prime$, we have
\[f_1^{\prime}(\alpha)=f_1^{\prime}(0)+\int_{0}^{\alpha} f_1^{\prime\prime}(\xi)d \xi \]
Using Theorem \ref{th:4} and that $f_1^{\prime}(0)=-2\delta(v)^2$ we have the following:
\begin{align*}
f_1^{\prime}(\alpha) & \leq -2\delta(v)^2+2\delta(v)^2\int_{0}^{\alpha} \psi_1^{\prime\prime}\left(\lambda_2(v)-2 \sqrt{2} \xi \delta(v)\right)d \xi, \\
& = -2\delta(v)^2-\frac{\delta(v)}{\sqrt{2}}\int_{0}^{\alpha} \psi_1^{\prime\prime}\left(\lambda_2(v)-2 \sqrt{2} \xi \delta(v)\right)d \left(\lambda_2(v)-2 \sqrt{2} \xi \delta(v)\right), \\
&=  -2\delta(v)^2-\frac{\delta(v)}{\sqrt{2}} \left(\psi_1^{\prime}\left(\lambda_2(v)-2 \sqrt{2} \xi \delta(v)\right)-\psi^{\prime}(\lambda_2(v))\right).
\end{align*}
Therefore, $f_1^{\prime}(\alpha)\leq 0$ is true if we have that:
\begin{align}
\label{eq:44}
-\psi^{\prime}\left(\lambda_2(v)-2 \sqrt{2} \alpha \delta(v)\right)+ \psi^{\prime\prime}\left(\lambda_2(v)\right)\leq 2 \sqrt{2}\delta(v).
\end{align}
\end{proof}

Following \cite{bai:2009}, at this point we also denote $\varrho: [0,\infty)\rightarrow [1,\infty )$ as the inverse function of $\psi(t), t \geq 1$ and $\rho: [0,\infty)\rightarrow (0,1]$ as the inverse function of $-\frac{1}{2}\psi^{\prime}(t), \ t \in (0,1]$. In other words, this can be written as in \eqref{in_other_words1}.
\begin{equation}
\label{in_other_words1}
\begin{aligned}
&\varrho(s)= t \Leftrightarrow \psi(t)=s \quad s\geq 0, t\geq 1, \\
&\rho(s)= t \Leftrightarrow -\psi^{\prime}(t)=2s \quad s\geq 0, 0<t\leq 1. 
\end{aligned}
\end{equation}

\begin{lemma}[Adapted from Lemma 4.3 in \cite{bai:2004}]
\label{lem:6}
The largest $\alpha$ that satisfies the condition of Theorem \ref{th:5} is given by equation \eqref{eq:45}. 
\begin{align}
\label{eq:45}
\bar{\alpha}:= \frac{\rho(\sqrt{2}\delta(v))-\rho(2\sqrt{2}\delta(v))}{2\sqrt{2}\delta(v)}.
\end{align}
\end{lemma}
\begin{proof}
The proof follows using the same arguments as in the original Lemma 4.3 in \cite{bai:2004}.
\end{proof}
\begin{theorem}[[Adapted from Lemma 4.4 in \cite{bai:2004}]
\label{th:7}
With $\bar{\alpha}$ defined in \eqref{eq:45} we have the following:
\begin{align}
\label{eq:46}
\bar{\alpha}\geq \frac{1}{\psi^{\prime\prime}\left(\rho(2\sqrt{2}\delta(v))\right)}= \hat{\alpha}
\end{align}
\end{theorem}
\begin{proof}
From the definition of $\rho$, we have:
\[-\psi^{\prime}(\rho(\sqrt{2}\delta(v)))=2\sqrt{2}\delta(v).\]
Differentiating both sides with respect to $\delta(v)$ and simplifying we get \eqref{eq:47}.
\begin{align}
\label{eq:47}
\rho^{\prime}(\sqrt{2}\delta(v))=\frac{2}{-\psi^{\prime\prime}\left(\rho(\sqrt{2}\delta(v))\right)}<0
\end{align}
This shows that $\rho$ is monotonically decreasing. From Lemma \ref{lem:6} we have:
\begin{align}
\label{eq:48}
\bar{\alpha}=\frac{1}{2\sqrt{2}\delta(v)} \int_{2\sqrt{2}\delta(v)}^{\sqrt{2}\delta(v)} \rho^{\prime}(\sigma) d \sigma=\frac{1}{\sqrt{2}\delta(v)} \int_{\sqrt{2}\delta(v)}^{2\sqrt{2}\delta(v)} \frac{d \sigma}{\psi^{\prime\prime}\left(\rho(\sqrt{2}\sigma)\right)},
\end{align}
after using \eqref{eq:47}. Since, we need a lower bound for $\bar{\alpha}$, we replace the argument of the last integral by its minimal value. Or, equivalently, we find the maximal value of $\psi^{\prime\prime}(\rho(\sqrt{2}\delta(v)))$ for $\sigma \in [\sqrt{2}\delta(v),2\sqrt{2}\delta(v)]$. From \eqref{eq:29}, we know that $\psi^{\prime\prime}$ is monotonically decreasing. This means that $\psi^{\prime\prime}(\rho(\sqrt{2}\delta(v)))$ is maximal for $ \sigma \in [\sqrt{2}\delta(v),2\sqrt{2}\delta(v)]$, when $\rho(\sqrt{2}\sigma$ is minimal. As $\rho$ is monotonically decreasing this happens when $\sigma= 2 \sqrt{2}\delta(v)$. Therefore, we have that
\[\bar{\alpha}=\int_{\sqrt{2}\delta(v)}^{2\sqrt{2}\delta(v)} \frac{d \sigma}{\sqrt{2}\delta(v) \psi^{\prime\prime}\left(\rho(\sqrt{2}\sigma)\right)} \ \geq \ \frac{1}{\psi^{\prime\prime}\left(\rho(2\sqrt{2}\delta(v))\right)},\] which proves the theorem.
\end{proof}

Hence, from now on in the paper, we will use:
\begin{align}\hat{\alpha}=\frac{1}{\psi^{\prime\prime}\left(\rho(2\sqrt{2}\delta(v))\right)}.\end{align}

\begin{lemma}[Lemma 4.5 in \cite{bai:2004}]
\label{lem:8}
If the step size $\alpha$ is such that $\alpha \leq \bar{\alpha}$ then
\begin{align}
\label{eq:49}
f(\alpha)\leq - \alpha \delta(v)^2.
\end{align}
\end{lemma}

Combining the results of Theorem \ref{th:7} and Lemma \ref{lem:8} we get the bound presented in Lemma \ref{lem:9}. 

\begin{lemma}
\label{lem:9}
 With $\bar{\alpha}$ being the step size defined earlier then
\begin{align}
\label{eq:50}
f(\hat{\alpha})\leq \frac{-\delta(v)^2}{\psi^{\prime\prime}\left(\rho(2\sqrt{2}\delta(v))\right)}.
\end{align}
\end{lemma}

\begin{theorem}
\label{th:10}
For $\delta(v)$ defined earlier the following lower bound holds:
\[\delta(v)\geq \frac{1}{2 \sqrt{2}}\psi^{\prime}\left(\varrho \left(2 \Psi(v)\right)\right)\]
\end{theorem}

\begin{proof}
Before proving this, we recall definitions 
For proving this theorem we need to recall some definitions we defined earlier. Using the definitions in \eqref{barrier1} (for $\Psi(\cdot)$) and \eqref{eq:25} (for the general case of $N>1$), we have \eqref{eq:51}.
\begin{align}
\label{eq:51}
2 \Psi(v)=2 \sum_{j=1}^{N}\Psi(v^j)=\sum_{j=1}^{N}\left[2\psi(\lambda_1(v^j))+\psi(\lambda_2(v^j))+\psi(\lambda_2(v^j))\right]:= \sum_{j=1}^{4N}\psi(z_j),
\end{align}
where $z_j$ is defined as a piece-wise linear function of $\lambda$'s given in \eqref{eq:52}.
\begin{align}
\label{eq:52}
z_j= 
  \begin{cases} 
   \lambda_1(v^j) & \text{if } 1\leq j\leq 2N, \\
   \lambda_2(v^j) & \text{if } 2N+1\leq j\leq 3N, \\
   \lambda_4(v^j) & \text{if } 3N+1\leq j\leq 4N.
  \end{cases}
\end{align}
Now, using \eqref{eq:52} with Lemma \ref{lem:28} we have that
\begin{align*}
\psi^{\prime}\left(\varrho \left(2 \Psi(v)\right)\right) & = \psi^{\prime}\left(\varrho \left(\sum_{j=1}^{4N}\psi(z_j)\right)\right) \ \leq \ \sqrt{\sum_{j=1}^{4N} \psi^{\prime}(z_j)^2}= \\
& = \sqrt{\sum_{j=1}^{N} \left[2\psi^{\prime}(\lambda_1(v^j))^2+\psi^{\prime}(\lambda_2(v^j))^2+\psi^{\prime}(\lambda_4(v^j))^2\right]}= \\
& = 2 \sqrt{2} \delta(v).
\end{align*}
Simplifying the last expression, we end the proof.
\end{proof}
Combining the results of Lemma \ref{lem:9} and Theorem \ref{th:10} we have that:
\begin{align}
\label{eq:53}
f(\hat{\alpha})\leq - \frac{\left[\psi^{\prime}\left(\varrho \left(2 \Psi(v)\right)\right)\right]^2}{8\psi^{\prime\prime}\left(\rho\left(\psi^{\prime}\left(\varrho \left(2 \Psi(v)\right)\right)\right)\right)}.
\end{align}

This inequality in \eqref{eq:53} presents an upper bound on the decrease in $\Psi(v)$ during an inner iteration. A variant version of this upper bound has been discussed in the literature previously, but for the original SOCO problem. The interested reader is referred to \cite{bai:2004,bai:2009}.

\begin{lemma}
\label{lem:11}
If $v \in \Upsilon_+$ and $\beta \geq 1$, then we have
\begin{align}
\label{eq:54}
\Psi(\beta v)\leq 2N \psi\left(\beta \varrho \left(\frac{\Psi(v)}{4N}\right)\right).
\end{align}
\end{lemma}
\begin{proof}
Using the definitions from \eqref{eq:51} and \eqref{eq:52} with Lemma \ref{lem:29} we have that:
\allowdisplaybreaks{
\begin{align*}
\Psi(\beta v)= \sum_{j=1}^{N}\Psi(\beta v^j)& =\frac{1}{2} \sum_{j=1}^{4N}\psi(\beta z_j) \\
& \leq \frac{4N}{2}  \psi\left(\beta \varrho \left(\frac{1}{4N}\sum_{i=1}^{4N} \psi(z_i)\right)\right) \\
& = 2N  \psi\left(\beta \varrho \left(\frac{\Psi(v)}{4N}\right)\right).
\end{align*}}
\end{proof}

\begin{cor}
\label{cor:11}
If $\Psi(v) \leq \tau $ and $v_+=\frac{v}{\sqrt{1-\theta}}$ with $0 \leq \theta \leq 1$ then we have:
\begin{align*}
\Psi(\beta v_+)\leq 2N \psi\left( \frac{\varrho \left(\tau /4N\right)}{\sqrt{1-\theta}}\right).
\end{align*}
\end{cor}
\begin{proof}
By taking $\beta=\sqrt{1-\theta}$ in Lemma \ref{lem:21} and using the fact that function $\varrho(s)$ is monotonically increasing, the above corollary follows immediately.
\end{proof}

We are ready to finish the subsection with the required iteration bound of Algorithm \ref{IPM_alg}. To do that, we count the number of times lines 6--8 are executed in the algorithm (we can also call those inner iterations) until, of course, the condition of line 5 ($\Psi(v):=\Phi(x,s,\mu)$ is less than or equal to $\tau$). In accordance to the literature (and particularly \cite{bai:2009}) we will refer to the value of $\Psi(v)$ during the first inner iteration (right after $\mu$ has been updated in line 4 during the current outer iteration) as $\Psi_0$. Every next iteration $i=1, \ldots, T$ (where $T$ marks the number of iterations/updates in the current outer iteration/$\mu$ value) will have $\Psi_i$. Note that due to \eqref{cor:11}, we have $\Psi_0\leq 2N \psi\left( \frac{\varrho \left(\tau /4N\right)}{\sqrt{1-\theta}}\right)$. Let $L=2N \psi\left( \frac{\varrho \left(\tau /4N\right)}{\sqrt{1-\theta}}\right)$, for simplicity.

Following the approaches in \cite{bai:2009} and \cite{bai:2004}, we assume the existence of constants $\kappa > 0$ and $\gamma \in (0,1]$ such that when $\Psi(v)> \tau$, the kernel function of type-2 SOCO satisfies (see the proven bound in equation \eqref{eq:53}):
\[\frac{\left[\psi^{\prime}\left(\varrho \left(2 \Psi(v)\right)\right)\right]^2}{8\psi^{\prime\prime}\left(\rho\left(\psi^{\prime}\left(\varrho \left(2 \Psi(v)\right)\right)\right)\right)} \geq \kappa \Psi(v)^{1-\gamma}.\]

\begin{lemma}(Lemma 3.15 in \cite{bai:2009}, Lemma 5.1 in \cite{bai:2004})
\label{lem:12}
\text{If} $K$ is the number of inner iterations between two successive outer iterations (or, $\mu$ updates), then
\[K\leq \frac{L^{\gamma}}{\kappa \gamma}.\]
\end{lemma}

The final upper bound on the number of iterations is found to be \[\frac{L^\gamma}{\theta \kappa \gamma} \log \frac{3N}{\epsilon},\] which is derived by multiplying $K$ from Lemma \ref{lem:12} with the number of barrier parameter updates (shown in Lemma II.17 in \cite{roos1997theory}) to be $\frac{1}{\theta} \log \frac{3N}{\epsilon}.$ Substituting $L$, we obtain 
\[\frac{1}{\theta \kappa \gamma} \left[2N \psi\left( \frac{\varrho \left(\tau /4N\right)}{\sqrt{1-\theta}}\right)\right]^\gamma\log \frac{3N}{\epsilon}.\]


\section{Concluding remarks and future work}
\label{sec:conclusion}

In this paper, we have discussed an extension of the traditional SOCO problem, called a type-$k$ SOCO problem. Specifically, we propose an extension and provide the appropriate theory for a primal-dual interior point algorithm when applied to a type-2 SOCO variant. This variant can be applied in certain facility location problems in which criteria such as $x_1\geq x_2$ arise \cite{lobo:1998,kuo:2004}. Such problems can be written in the type-2 SOCO format presented here. The solution approach investigated here is no different than the generic version presented for solving traditional SOCO problems. However, due to the different framework, the resulting iteration bound derived is different than the one obtained in \cite{bai:2009}. 

A possible future avenue for this research would to discuss the seven eligible kernel functions available in literature and adapt them for type-2 SOCO problems. Moreover, one can also consider applying the recently proposed Nesterov acceleration schemes in the context of affine scaling method \cite{morshed:2020} and sampling kaczmarz method \cite{morshed:2019} to general IPMs, which may increase the efficiency of the proposed primal dual barrier method in this work as suggested by recent work \cite{morshed:2020}. We will also investigate the generalized type-$k$ SOCO problem and introduce the required Jordan product, from which the results of regular SOCO and type-2 SOCO should follow as special cases. We consider the cone in \eqref{typekSOCO} (shown in Section \ref{sec:intro}) as the generalized type-$k$ SOCO.






\bibliographystyle{unsrt}
\bibliography{IPM_Soco}


\section*{Appendix}

In the appendix, we provide some technical lemmata and theorems.
\begin{lemma}
\label{lem:13}
For all $x \in \R^n$, we have that $|\lambda_2| \leq 2\|x\|$ and  $|\lambda_4|\leq 2\|x\|$.
\end{lemma}

\begin{proof}
By definition, we have that
\[2\lambda^2_1(x)+\lambda^2_2(x)+\lambda^2_4(x)=4\|x\|^2.\]
Since $\lambda^2_2(x)\leq 2\lambda^2_1(x)+\lambda^2_2(x)+\lambda^2_4(x)=4\|x\|^2$ and $\lambda^2_4(x) \leq 2\lambda^2_1(x)+\lambda^2_2(x)+\lambda^2_4(x)=4\|x\|^2$, we immediately get the lemma.
\end{proof}

\begin{lemma}
\label{lem:14}
Let $x,s \in \R^n$, then $ \lambda_2(x+s) \geq \lambda_2(x)-2  \|s\|$.
\end{lemma}
\begin{proof}
The proof follows immediately after Lemma \ref{lem:13}.
\end{proof}

\begin{lemma}
\label{lem:15}
Let $x,s,t\in R^n$. Then:
\begin{align}
\label{eq:55}
\textbf{Tr}\left((x\diamond s)\diamond t \right)=\textbf{Tr}\left(x\diamond (s\diamond t) \right).
\end{align}
\end{lemma}

\begin{proof}
$R(x)$ is symmetric (by design) and we know that $x \diamond s=R(x)s=R(s)x$. This leads to:
\[\textbf{Tr}\left((x\diamond s)\diamond t \right)=2(x\diamond s)^T t=2(R(s) x)^T t=2x^T(s\diamond t)=\textbf{Tr}\left(x\diamond (s\diamond t) \right),\]
which shows the lemma.
\end{proof}

\begin{lemma}
\label{lem:16}
Let $x,s\in R^n$. Then:
\begin{align*}
  \frac{1}{2}\left[\lambda_2(x)\lambda_4(s)+\lambda_4(x)\lambda_2(s)\right] & \leq \textbf{Tr}(x\diamond s)-\lambda_1(x)\lambda_1(s) \\ 
& \leq \frac{1}{2}\left[\lambda_2(x)\lambda_2(s)+\lambda_4(x)\lambda_4(s)\right].  
\end{align*}
\end{lemma}
\begin{proof}
Considering the definitions given in \eqref{eq:7} we have the following:
\begin{multline*}
\lambda_1(x)\lambda_1(s)+\frac{1}{2}\left[\lambda_2(x)\lambda_4(s)+\lambda_4(x)\lambda_2(s)\right]= 2\left(x_1s_1+x_2s_2-\|x_{3:n}\||s_{3:n}\|\right) \\
\leq 2x^Ts = \textbf{Tr}(x\diamond s).
\end{multline*}

\noindent Conversely, from the other direction of the Cauchy-Schwarz inequality, we have:
\begin{multline*}
\lambda_1(x)\lambda_1(s)+\frac{1}{2}\left[\lambda_2(x)\lambda_2(s)+\lambda_4(x)\lambda_4(s)\right]= 2\left(x_1s_1+x_2s_2+\|x_{3:n}\||s_{3:n}\|\right) \\
\geq 2x^Ts = \textbf{Tr}(x\diamond s).
\end{multline*}
This finishes the proof.
\end{proof}

\begin{cor}
\label{cor:16}
Let $x\in R^n$ and $s \in \Upsilon^n$. Then:
\[\lambda_1(x)\textbf{Tr}(s)\leq \textbf{Tr}(x\diamond s)\leq\lambda_4(x)\textbf{Tr}(s).\]
\end{cor}
\begin{proof}
Because $s \in \Upsilon^n$, we have that $\lambda_1(s), \lambda_4(s)\geq 0$. Now, applying Lemma \ref{lem:16}, we can see that this corollary also holds.
\end{proof}

\begin{lemma}
\label{lem:17}
Let $x,s\in R^n$. Then:
\[\overline{\textbf{det}}(x\diamond s) \leq \overline{\textbf{det}}(x)\overline{\textbf{det}}(s),\]
with the equality holding when the vectors $x_{3:n}$ and $s_{3:n}$ are linearly dependent.
\end{lemma}
\begin{proof}
By definition, we have that
\begin{align*}
& \overline{\textbf{det}}(x\diamond s) \\  
& = \left(x^Ts+x_2s_1+x_1s_2+x_{3:n}^Ts_{3:n}\right)^2-2\|(x_1+x_2)s_{3:n}+(s_1+s_2)x_{3:n}\|^2 \\
&=(x_1+x_2)^2(s_1+s_2)^2+4(x_{3:n}^Ts_{3:n})^2-2(x_1+x_2)^2\|s_{3:n}\|^2 \\
& \quad \quad -2(s_1+s_2)^2\|x_{3:n}\|^2\\
& \leq (x_1+x_2)^2 \left\{(s_1+s_2)^2-2\|s_{3:n}\|^2\right\}-2\|x_{3:n}\|^2\left\{(s_1+s_2)^2-2\|s_{3:n}\|^2\right\}\\
&=\left\{(x_1+x_2)^2-2\|x_{3:n}\|^2\right\}\left\{(s_1+s_2)^2-2\|s_{3:n}\|^2\right\} = \overline{\textbf{det}}(x)\overline{\textbf{det}}(s),
\end{align*}
and equality holds if and only if $|x^T_{3:n}s_{3:n}|=\|x_{3:n}\|\|s_{3:n}\|$. This means that the equality holds if the vectors $x_{3:n}$ and $s_{3:n}$ are linearly dependent.
\end{proof}

\begin{lemma}
\label{lem:18}
Let $x,s\in R^n$. Then:
\[\textbf{det}(x\diamond s)\leq \textbf{det}(x)\textbf{det}(s)+\underline{\textbf{det}}(x)\underline{\textbf{det}}(s)\]
\[\underline{\textbf{det}}(x\diamond s)\leq \textbf{det}(x)\underline{\textbf{det}}(s)+\underline{\textbf{det}}(x)\textbf{det}(s).\]
\end{lemma}

\begin{proof}
Using the same idea as the proof of Lemma \ref{lem:17} and the Cauchy-Schwarz inequality, this lemma also follows.
\end{proof}

\begin{lemma}
\label{lem:19}
Let $\psi:\R_{++}\rightarrow \R_{+}$ and $x \in \Upsilon^n_+$.  Then, we have that $\psi(x) \in \Upsilon^n$.
\end{lemma}

\begin{proof}
Since $x \in \Upsilon^n_+$, all of the eigenvalues are positive. Hence, by definition, we must have that $\psi(\lambda_1(x)),\psi(\lambda_2(x)),\psi(\lambda_4(x))$ exist and are non-negative. Since $v_1,v_2,v_4 \in \Upsilon^n$, then:
\[\psi(x)=\psi(\lambda_1(x))v_1+\psi(\lambda_2(x))v_2+\psi(\lambda_4(x))v_4 \in \Upsilon^n.\]
This proves the lemma.
\end{proof}

Assuming $\psi(t)$ is twice differentiable, then the derivatives $\psi^{\prime}(t)$ and $\psi^{\prime\prime}(t)$ exist for all $t>0$. Similarly, $\psi^{\prime}(x)$ and $\psi^{\prime\prime}(x)$ for vectors $x$ exist and are:
\begin{align}
&\psi^{\prime}(x):=\psi^{\prime}(\lambda_1(x))v_1+\psi^{\prime}(\lambda_2(x))v_2+\psi^{\prime}(\lambda_4(x))v_4 \label{eq:56} \\
&\psi^{\prime\prime}(x):=\psi^{\prime\prime}(\lambda_1(x))v_1+\psi^{\prime\prime}(\lambda_2(x))v_2+\psi^{\prime\prime}(\lambda_4(x))v_4 \label{eq:57}
\end{align}
Now we will investigate two cases. When $x_{3:n}\neq 0$, we can calculate the following:
\begin{align}
\label{eq:58}
\|\psi(x)\|=\frac{1}{2}\sqrt{2\psi^2(\lambda_1(x))+\psi^2(\lambda_2(x))+\psi^2(\lambda_4(x))}.
\end{align}
Note that \eqref{eq:14} also holds when $\|x_{3:n}\|=0$. When $x_{3:n}=0$, we have $\lambda_1(x)=x_1-x_2, \ \lambda_2(x)=x_1+x_2, \ \lambda_4(x)=x_1+x_2$ and 
\begin{align}
\label{eq:58a}
\|\psi(x)\|=\frac{1}{\sqrt{2}}\sqrt{\psi^2(\lambda_1(x))+\psi^2(\lambda_2(x))}.
\end{align}

\begin{lemma}
\label{lem:20}
$\Psi(x)$ is strictly convex for all $x \in \Upsilon_+^n$.
\end{lemma}
\begin{proof}
We need to show that for all $x, s \in \Upsilon_+^n, x\neq s$, \eqref{eq:59} holds.
\begin{align}
\label{eq:59}
\Psi\left(\frac{x+s}{2}\right) \leq \frac{1}{2}\Psi(x)+\frac{1}{2} \Psi(s).
\end{align}
First, for all $x, s \in \Upsilon_+^n$ we have:
\begin{align*}
& \lambda_1 \left(\frac{x+s}{2}\right)= \frac{1}{2}\lambda_1(x)+\frac{1}{2}\lambda_1(s), \\
& \lambda_4 \left(\frac{x+s}{2}\right)\leq \frac{1}{2}\lambda_4(x)+\frac{1}{2}\lambda_4(s), \\ 
& \lambda_2 \left(\frac{x+s}{2}\right)\geq \frac{1}{2}\lambda_2(x)+\frac{1}{2}\lambda_2(s).
\end{align*}
Since $\psi(t)$ is strictly convex for $t > 0$, we have:
\begin{align*}
 \Psi\left(\frac{x+s}{2}\right) & \leq \frac{1}{2} [\lambda_1(x)+\lambda_1(s)]+\frac{1}{4}[\lambda_2(x)+\lambda_2(s)+\lambda_4(x)+\lambda_4(s] \\ 
 & = \frac{1}{2}\Psi(x)+\frac{1}{2}\Psi(s)  
\end{align*}
This finishes the proof that $\Psi(x)$ is strictly convex for $x \in \Upsilon_+^n$.
\end{proof}

\begin{lemma}
\label{lem:21}
For any vector $a\in \R^n$ and for $\beta \geq 0$, we have:
\begin{align}
\label{eq:60}
\left(I_n+\beta aa^T\right)^{\frac{1}{2}}=I_n+\frac{\beta aa^T}{1+\sqrt{1+\beta a^Ta}}.
\end{align}
\end{lemma}

\begin{proof}
Let $A=I_n+\beta aa^T$. The eigenvalues of $A$ are $1+\beta a^Ta$ with a multiplicity of 1 and $1$ with a multiplicity of $n-1$. The corresponding eigenvalues are $\frac{aa^T}{a^T a}$ and $(e_2,e_3,...,e_n)$, respectively. Since $A$ is a Hermitian matrix, there exists a unique representation of $A$ in terms of the eigenvectors of $A$:
\[A=\left(1+\beta aa^T\right)\frac{aa^T}{a^T a}+\sum\limits_{i=2}^{n}e_i.\]
Therefore, we have that
\[A^{\frac{1}{2}} = \sqrt{1+\beta aa^T} \ \frac{aa^T}{a^T a}+A-\left(1+\beta aa^T\right)\frac{aa^T}{a^T a}=I_n+\frac{\beta aa^T}{1+\sqrt{1+\beta a^Ta}}\]
;which proves the lemma.
\end{proof}

\begin{theorem}
\label{th:22}
Let matrix $W \succ 0$ be such that $WQW=\lambda Q$ for some $\lambda >0$. Then, there exists $a=(a_1,a_2;\bar{a})\in \Upsilon$ with $\textbf{det}(a)= \bar{\textbf{det}}(a)=1$ such that $W=\sqrt{\lambda}W_a$. In addition we have, $W^{-1}=\frac{1}{\sqrt{\lambda}}W_a$, where $W_a$ is given by:
\[W_a=\begin{bmatrix}
    a_{1} & a_{2} & \bar{a}^T \\
    a_{2} & a_{1} & \bar{a}^T \\
    \bar{a} & \bar{a} & I_{n-2}+\frac{2 \bar{a}\bar{a}^T}{1+a_1+a_2} 
\end{bmatrix}.\]
\end{theorem}

\begin{proof}
First, assume $\lambda=1$. Since $Q^2=I_n, WQW=Q$, this implies that $QWQ=W^{-1}$. Without loss of generality, we can assume $W$ has the following form:
\[W=\begin{bmatrix}
    a_{1} & a_{2} & \bar{a}^T \\
    a_{2} & a_{1} & \bar{a}^T \\
    \bar{a} & \bar{a} & C 
\end{bmatrix},\]
where $C$ is a symmetric matrix. Hence, we get that
\[W^{-1}=QWQ=\begin{bmatrix}
    a_{1} & a_{2} & -\bar{a}^T \\
    a_{2} & a_{1} & -\bar{a}^T \\
    -\bar{a} & -\bar{a} & C 
\end{bmatrix}.\]
Considering the condition that $I=WW^{-1}$, we have:
\[\begin{bmatrix}
    1 & 0 & 0 \\
    0 & 1 & 0 \\
    0 & 0 & I 
\end{bmatrix}=\begin{bmatrix}
    a_{1}^2+a_{1}^2- \|\bar{a}\|^2 & 2a_1a_{2}-\|\bar{a}\|^2 & (a_1+a_2)\bar{a}^T-\bar{a}^TC \\
    2a_1a_{2}-\|\bar{a}\|^2 & a_{1}^2+a_{1}^2- \|\bar{a}\|^2 & (a_1+a_2)\bar{a}^T-\bar{a}^TC \\
    -(a_1+a_2)\bar{a}+C\bar{a} & -(a_1+a_2)\bar{a}+C\bar{a} & C^2-2\bar{a}\bar{a}^T 
\end{bmatrix}.\]
This identity holds if and only if the following relations of \eqref{eq:61a}--\eqref{eq:61d} holds:
\begin{align}
\label{eq:61a}
&a_{1}^2+a_{1}^2- \|\bar{a}\|^2=1, \\
\label{eq:61b}
&2a_1a_{2}-\|\bar{a}\|^2=0, \\
\label{eq:61c}
&-(a_1+a_2)\bar{a}+C\bar{a}=\textbf{0}, \\
\label{eq:61d}
&C^2-2\bar{a}\bar{a}^T=I_n.
\end{align}
These relations give us that: $ \textbf{det}(a)=1; \ \overline{\textbf{det}}(a) = 1$. 
Similarly, if we consider the last relation of the above equation along with Lemma \ref{lem:21} we have:
\[C=I_{n-2}+\frac{2\bar{a}\bar{a}^T}{1 \pm \sqrt{1+2\|\bar{a}\|^2}}.\]
Recalling \eqref{eq:61c}, we deduce that:
\begin{align}
\label{eq:62}
C=I_{n-2}+\frac{2\bar{a}\bar{a}^T}{1+\sqrt{1+2\|\bar{a}\|^2}}.
\end{align}
This completes the proof for $\lambda=1$. In the case when $\lambda \neq 1$, a simple multiplication of $W_a$ by $\sqrt{\lambda}$ gives the result.
\end{proof}


\begin{theorem}
\label{th:23}
Let $x,s \in \Upsilon_+$. Then, there exists a unique automorphism $W$ of $\Upsilon$ such that $Wx=W^{-1}s$. Using the same notation as before, the automorphism is given by $W=\sqrt{\lambda}W_a$, where
\begin{align}
& \quad \lambda=\frac{a^TQs}{a^Tx}=\frac{s_1-s_2}{x_1-x_2}=\sqrt{\frac{\textbf{det}(s)-\underline{\textbf{det}}(s)}{\textbf{det}(x)-\underline{\textbf{det}}(x)}}=\frac{\lambda_1(s)}{\lambda_1(x)} \label{eq:63} \\
& a= \frac{2\beta(x_1-x_2)}{(\alpha^2-\beta^2)}
\begin{bmatrix}
   -1 \\
   1 \\
\textbf{0} 
\end{bmatrix}+\frac{2\beta(x_1-x_2)}{\lambda(\alpha^2-\beta^2)}(s+\lambda Qx), \nonumber
\end{align}
with $\alpha= 2a^Tx, \beta = 2(a^Tx-x_1+x_2)$
\end{theorem}

\begin{proof}
First, let us define matrices $P, \overline{P}$ as
\[P=QPQ=\begin{bmatrix}
    0 & 1 & 0 \\
    1 & 0 & 0 \\
    0 & 0 & I 
\end{bmatrix}, \quad \overline{P}=QP=\begin{bmatrix}
    0 & 1 & 0 \\
    1 & 0 & 0 \\
    0 & 0 & -I 
\end{bmatrix}.\]
By Theorem \ref{th:22}, every automorphism of $\Upsilon$ has the form $W=\sqrt{\lambda}W_a$ with $1\neq \lambda > 0$ and $a=(a_1,a_2;\bar{a})\in \Upsilon$ and $\textbf{det}(a)= \overline{\textbf{det}}(a)=1$. Also, whenever $\textbf{det}(a)= \overline{\textbf{det}}(a)=1$ holds, we have the following hold, too:
\begin{align*}
 a_1-a_2 \ & = \ \sqrt{a_1^2+a_2^2-2a_1a_2} \\ 
 & = \sqrt{2(a_{1}^2+a_{1}^2- \|\bar{a}\|^2)-((a_1+a_{2})^2-2\|\bar{a}\|^2)} =1.   
\end{align*}
Since, $Wx=W^{-1}s$ holds if and only if $W^2x=s$, we need to find $a$ and $\lambda$ such that $W^2x=s$. After some calculations, we have that
\[W^2=\lambda W_a^2=\lambda \begin{bmatrix}
    a_{1}^2+a_{2}^2+ \|\bar{a}\|^2 & 2a_1a_{2}+\|\bar{a}\|^2 & 2(a_1+a_2)\bar{a}^T \\
    2a_1a_{2}+\|\bar{a}\|^2 & a_{1}^2+a_{2}^2+ \|\bar{a}\|^2 & 2(a_1+a_2)\bar{a}^T \\
    2(a_1+a_2)\bar{a} & 2(a_1+a_2)\bar{a}& I_{n-2}+4\bar{a}\bar{a}^T 
\end{bmatrix}.\]
Using the given fact that $\textbf{det}(a)=a_{1}^2+a_{1}^2- \|\bar{a}\|^2= \overline{\textbf{det}}(a)=(a_1+a_{2})^2-2\|\bar{a}\|^2=1$, we have the following relation:
\[W^2=\lambda \begin{bmatrix}
    2(a_{1}^2+a_{2}^2)-1 & 4a_1a_{2}& 2(a_1+a_2)\bar{a}^T \\
    4a_1a_{2}& 2(a_{1}^2+a_{2}^2)-1& 2(a_1+a_2)\bar{a}^T \\
    2(a_1+a_2)\bar{a} & 2(a_1+a_2)\bar{a}& I_{n-2}+4\bar{a}\bar{a}^T 
\end{bmatrix}.\]
Furthermore, denoting $x=(x_1,x_2;\bar{x})$ and $s=(s_1,s_2;\bar{s})$, from the above reasoning our goal is to find $\lambda$ and $a$ such that the following holds:
\[\lambda \begin{bmatrix}
    2(a_{1}^2+a_{2}^2)-1 & 4a_1a_{2}& 2(a_1+a_2)\bar{a}^T \\
    4a_1a_{2}& 2(a_{1}^2+a_{2}^2)-1& 2(a_1+a_2)\bar{a}^T \\
    2(a_1+a_2)\bar{a} & 2(a_1+a_2)\bar{a}& I_{n-2}+4\bar{a}\bar{a}^T 
\end{bmatrix}\begin{bmatrix}
    x_{1}\\
    x_{2}\\
    \bar{x} 
\end{bmatrix}=\begin{bmatrix}
    s_{1}\\
    s_{2}\\
    \bar{s} 
\end{bmatrix}.\]
This matrix relation is equivalent to the following system in \eqref{eq:64}.
\begin{align}
\label{eq:64}
&(2a_{1}^2+2a_{2}^2-1 )x_1+4a_1a_2x_2+2(a_1+a_2)\bar{a}^T\bar{x}=\frac{s_1}{\lambda}, \nonumber \\
&( 2a_{1}^2+2a_{2}^2-1 )x_2+4a_1a_2x_1+2(a_1+a_2)\bar{a}^T\bar{x}=\frac{s_2}{\lambda}, \\
& 2(a_1+a_2)x_1\bar{a}+2(a_1+a_2)x_2\bar{a}+\bar{x}+4\bar{a}^T\bar{x}\bar{a}=\frac{\bar{s}}{\lambda}. \nonumber
\end{align}
Since $a^Tx=a_1x_1+a_2x_2+\bar{a}^T\bar{x}$ and $a_1-a_2=1$, the system of equations \eqref{eq:64} can be transformed into the following equations:
\begin{equation*}
    \begin{aligned}
        &\frac{s_1}{\lambda}+x_1=2a^Tx(a_1+a_2)-2a_2(x_1-x_2),\\
        &\frac{s_2}{\lambda}+x_2=2a^Tx(a_1+a_2)-2a_1(x_1-x_2),\\
        &\frac{\bar{s}}{\lambda}-\bar{x}=4a^Tx\bar{a}-2(x_1-x_2)\bar{a}.
    \end{aligned}
\end{equation*}
Converting them to a linear system format, we have the system of \eqref{eq:65}.
\begin{align}
\label{eq:65}
\begin{bmatrix}
   \frac{s_1}{\lambda}+x_1\\
   \frac{s_2}{\lambda}+x_2\\
    \frac{\bar{s}}{\lambda}-\bar{x} 
\end{bmatrix}=\frac{s+\lambda Qx}{\lambda}&=2a^Tx\begin{bmatrix}
    a_{1}\\
    a_{2}\\
    \bar{a} 
\end{bmatrix}+2a^Tx\begin{bmatrix}
    a_{2}\\
    a_{1}\\
    \bar{a} 
\end{bmatrix}-2(x_1-x_2)\begin{bmatrix}
    a_{2}\\
    a_{1}\\
    \bar{a} 
\end{bmatrix}= \nonumber \\
&=2a^Txa+2a^TxPa-2(x_1-x_2)Pa.
\end{align}
So far, we have shown that if $W=\sqrt{\lambda}W_a$ and if $W^2x=s$, then $a$ satisfies the above relation. Similarly, since we know that $W^{-1}=\frac{1}{\sqrt{\lambda}}W_{Qa}$ and $W^{-2}s=x$ we can exchange $x$ with $s$, $a$ with $Qa$ and $\lambda$ with $\frac{1}{\lambda}$ to get the following identity:
\[\frac{x+\frac{1}{\lambda} Qs}{\frac{1}{\lambda}}
=2(Qa)^TsQa+2(Qa)^TsPQa-2(s_1-s_2)PQa.\]
After further simplification it becomes as in \eqref{eq:66}.
\begin{align}
\label{eq:66}
\lambda x+Qs=2a^TQsQa+2a^TQs\bar{P}a-2(s_1-s_2)\bar{P}a.
\end{align}
Now, taking the inner product on both sides of \eqref{eq:66} with $a$, and using the simple identities that $a^TQa=\textbf{det}(a)=1$ and $a^T\bar{P}a=0$ we can deduce the following:
\[\lambda=\frac{a^TQs}{a^Tx}.\]
Multiplying both sides of equation \eqref{eq:65} with $P$ and subtracting it from \eqref{eq:65} we get that:
\[\frac{(I-P)s+\lambda(Q-\bar{P})x}{\lambda}=2(x_1-x_2)(I-P)a.\]
After some calculations, we can simplify this even further as:
\[\frac{1}{\lambda}\begin{bmatrix}
   s_1-s_2\\
   s_2-s_1\\
    \textbf{0} 
\end{bmatrix}+
\begin{bmatrix}
   x_1-x_2\\
   x_2-x_1\\
    \textbf{0} 
\end{bmatrix}=
2\begin{bmatrix}
   x_1-x_2\\
   x_2-x_1\\
    \textbf{0} 
\end{bmatrix}.\]
After simplifying, we have the following closed form expression for $\lambda$:
\[\lambda=\frac{s_1-s_2}{x_1-x_2}.\]
Now, let us multiply by $P$ both sides of equation \eqref{eq:65} from the left. Using the fact that $P^2=I$, we get:
\begin{align*}
&\frac{1}{\lambda}(Ps+\lambda PQx)=2a^TxPa+2(a^Tx-x_1+x_2)a \implies \\
 \implies & a=\frac{1}{\lambda}\left(2a^TxP+2(a^Tx-x_1+x_2)I\right)^{-1}(Ps+\lambda PQx) \implies \\
\implies & a=\frac{1}{\lambda(\alpha^2-\beta^2)}\left(\alpha Ps-\beta s\right)+\frac{1}{(\alpha^2-\beta^2)}(\alpha QPs-\beta Qx) \implies \\
\implies  & a=\frac{2\beta(x_1-x_2)}{(\alpha^2-\beta^2)}\begin{bmatrix}
   -1\\
   1\\
    \textbf{0} 
\end{bmatrix}+\frac{2\beta(x_1-x_2)}{\lambda(\alpha^2-\beta^2)}(s+\lambda Qx),
\end{align*}
where $\alpha= 2a^Tx$ and $\beta = 2(a^Tx-x_1+x_2)$. This completes the proof.
\end{proof}

\begin{lemma}
\label{lem:24}
For $W, Q, \overline{Q}, \overline{P}$ defined earlier we have the following expressions:
\[WQW= \lambda Q ; \quad W \overline{Q}W= \lambda  \overline{Q} ; \quad W \overline{P}W= \lambda  \overline{P} \quad and \quad W^{-1} \overline{Q}W^{-1}= \frac{1}{\lambda} \overline{Q}.\]
\end{lemma}

\begin{proof}
The proof of this lemma follows directly from Theorem \ref{th:22} and \ref{th:23} with the usual definitions of $P, \ Q, \ \overline{P}$ and $\overline{Q}$ as those were given in the previous parts.
\end{proof}

\begin{lemma}
\label{lem:25}
Let $v:= \frac{Wx}{\sqrt{\mu}}\quad \left(=\frac{W^{-1}s}{\sqrt{\mu}}\right)$ and $\lambda$ as defined before. Then, we have that
\[\lambda= \sqrt{\frac{\overline{\textbf{det}}(s)}{\overline{\textbf{det}}(x)}}.\]
\end{lemma}

\begin{proof}
By definition
\[\overline{\textbf{det}}(v)=\frac{1}{\mu}x^TW \overline{Q} Wx=\frac{\lambda}{\mu}x^T \overline{Q} x=\frac{1}{\mu}s^TW^{-1} \overline{Q}W^{-1}s=\frac{1}{\lambda \mu}s^T \overline{Q} s.\]
After some simplification, this gives us that:
\[\lambda^2=\frac{s^T \overline{Q} s}{x^T \overline{Q} x}=\frac{\overline{\textbf{det}}(s)}{\overline{\textbf{det}}(x)},\]
which proves the lemma.
\end{proof}

\begin{lemma}
\label{lem:26}
For any $x,s \in \R^n$, we have the following:
\begin{itemize}
\itemsep0em
\item $\textbf{Tr}(\bar{x}\diamond \bar{s})=\textbf{Tr}(x\diamond s)$; 
\item $\textbf{det}(\bar{x})=\lambda \textbf{det}(x), \quad \textbf{det}(\bar{s})=\lambda^{-1} \textbf{det}(s), \quad \text{where} \quad \lambda=\frac{\overline{\textbf{det}}(s)}{\overline{\textbf{det}}(x)}$;
\item $x \in \Upsilon^n,(x \in \Upsilon^n_{+}) \Leftrightarrow \bar{x} \in \Upsilon^n,(\bar{x} \in \Upsilon^n_{+})$.
\end{itemize}
\end{lemma}

\begin{proof}
The proof of the first part is straightforward:
\[\textbf{Tr}(\bar{x}\diamond \bar{s})=\textbf{Tr}(Wx\diamond W^{-1}s)=2(Wx)^T(W^{-1}s)=2x^Ts=\textbf{Tr}(x\diamond s).\]
For the second part, we first define the matrix of \eqref{eq:67}.
\begin{align}
\label{eq:67}
Q=\text{diag}(1,1,-1,...,-1)\in \R^{n\times n},
\end{align}
which implies that $\text{det}(x)=x^TQx $. Now, using that $WQW=\lambda Q$, we have:
\[\textbf{det}(\bar{x})=\bar{x}^TQ\bar{x}=x^TWQWx=\lambda x^TQx=\lambda \textbf{det}(x). \]

Similarly, we can prove that $\textbf{det}(\bar{s})=\lambda^{-1} \textbf{det}(s)$. And finally, since $W$ is a automorphism of $\Upsilon^n$, we can deduce that $x \in \Upsilon^n$ if and only if $\bar{x} \in \Upsilon^n$. From the second point of the Lemma, we have that $\textbf{det}(x)>0$ if and only if $\textbf{det}(\bar{x})>0$; also $x \in \Upsilon^n_+$ if and only if $\bar{x} \in \Upsilon^n_+$. Therefore, at this point, the proof is complete.
\end{proof}

\begin{lemma}
\label{lem:27}
The following holds: $\lambda_1(v)=\frac{1}{\sqrt{\mu}}[\lambda_1(x)\lambda_1(s)]^{\frac{1}{2}}$.
\end{lemma}

\begin{proof}
By definition, we have:
\[\lambda_1(v)= \sqrt{v^TQv-v^T \overline{P}v}=\frac{\sqrt{\lambda}}{\sqrt{\mu}}[\textbf{det}(x)-\underline{\textbf{det}}(x)]^{\frac{1}{2}}=\frac{1}{\sqrt{\mu}}[\lambda_1(x)\lambda_1(s)]^{\frac{1}{2}},\]
which proves the lemma.
\end{proof}

\begin{lemma}[Lemma 3.10 in \cite{bai:2009}]
\label{lem:28}
For any vector, we have: $z \in \R^p$ 
\[\sqrt{\sum_{i=1}^{p} \psi^{\prime}(z_i)^2} \ \geq \ \psi^{\prime}\left(\varrho \left(\sum_{i=1}^{p} \psi(z_i)\right)\right).\]
\end{lemma}

\begin{lemma}[Lemma 3.12 in \cite{bai:2009}]
\label{lem:29}
For any vector $z \in \R^p$ and $\beta \geq 1$, we have:
\[\sum_{i=1}^{p} \psi(\beta z_i)\leq p \psi\left(\beta \varrho \left(\frac{1}{p}\sum_{i=1}^{p} \psi(z_i)\right)\right).\]
\end{lemma}

\end{document}